\documentclass{amsart}

\usepackage[latin1]{inputenc}
\usepackage{amsthm,amssymb, amsmath}
\usepackage[all]{xy}
\usepackage{graphicx}
\usepackage{subfig}

\newcommand{\texttit}{\textit}
\newcommand{\textresto}{\textup}

\newtheorem{theorem}{Theorem}[section]
\newtheorem{definition}{Definition}[section]
\newtheorem{lemma}[theorem]{Lemma}
\newtheorem{proposition}[theorem]{Proposition}
\newtheorem{corollary}[theorem]{Corollary}

\newtheorem*{remark}{Remark}
\DeclareMathOperator*{\var}{Var}

\begin{document}

\title{The entropy of Nakada's $\alpha$-continued fractions: analytical results}
\author{Giulio Tiozzo}

\begin{abstract}
We study the ergodic theory of a one-parameter family of interval maps $T_\alpha$ arising from generalized continued fraction algorithms. 
First of all, we prove the dependence of the metric entropy of $T_\alpha$ to be H\"older-continuous in the parameter $\alpha$.
Moreover, we prove a central limit theorem for possibly unbounded observables whose bounded variation 
grows moderately. This class of functions is large enough to cover the case of Birkhoff averages converging to the entropy.
\end{abstract}

\maketitle

\begin{center} 
\end{center}

\section{Introduction}

Let $\alpha \in [0,1]$. Let us define the map $T_{\alpha} : [\alpha-1, \alpha] \rightarrow [\alpha-1, \alpha] $ as $T_\alpha(0) = 0$ and 
\begin{equation} \label{definition}
T_{\alpha}(x) = \frac{1}{|x|} - a_{\alpha}(x)
\end{equation}
with $a_{\alpha}(x) := \left\lfloor \frac{1}{|x|} + 1 - \alpha \right\rfloor$. 
These systems were introduced by Nakada \cite{Nakada81} and are known in the literature as \emph{$\alpha$-continued fractions}, or 
 \emph{Japanese continued fractions}. 
By taking $x_{n, \alpha} = T_{\alpha}^n(x)$, $a_{n, \alpha} = a_{\alpha}(x_{{n-1}, \alpha})$, $\epsilon_{n, \alpha} = \textup{Sign}(x_{n-1, \alpha})$, 
the orbit under $T_\alpha$ generates the generalized continued fraction expansion 
$$ x = a_{0, \alpha}  + \frac{\epsilon_{1, \alpha}}{a_{1,\alpha} + \frac{\epsilon_{2,\alpha}}{a_{2, \alpha} +_{\ddots}}} $$ 


The algorithm, analogously to the Gauss map in the classical case, provides rational approximations of real numbers. 
It is known that for each $\alpha \in (0,1]$ there exists a unique invariant measure $\mu_{\alpha}(dx) = \rho_{\alpha}(x)dx$ absolutely 
continuous w.r.t. Lebesgue measure, and this measure is ergodic (see \cite{LuzziMarmi}). In this paper we will focus on the metric entropy 
of the $T_{\alpha}$'s, which is given by Rohlin's formula (see \cite{Rohlin})

\begin{equation} \label{Rohlin}
h(T_{\alpha}) = \int_{\alpha-1}^{\alpha} \log |T_\alpha'| d\mu_\alpha
\end{equation}

Nakada \cite{Nakada81} computed exact values of $h(T_\alpha)$ for $\alpha \geq \frac{1}{2}$, showing 
that in this interval $h(T_\alpha)$ is continuous, and smooth except for the point $\alpha = \frac{\sqrt{5}-1}{2}$, where the left and right
 derivatives do not coincide. 

In \cite{LuzziMarmi}, Luzzi and Marmi studied the behaviour of $h(T_\alpha)$ as a function of $\alpha$ for all parameters $\alpha \in (0,1)$.
 They gave numerical evidence that this function is continuous but not smooth. Nakada and Natsui 
\cite{NakadaNatsui} then proved that this function is indeed not monotone, by giving explicit construction of infinitely many monotonicity 
intervals. An extensive numerical study of these intervals has been carried out in \cite{CMPT}, and a complete characterization of 
all monotonicity intervals is given in \cite{CT}. Even though the entropy is conjecturally smooth on any such interval, there are 
points at which it is not even locally monotone, and the set of bifurcation parameters has a complicated 
self-similar structure. 

The first major result of this paper is the

\begin{theorem} \label{Holdermain}
The entropy function $\alpha \mapsto h(T_\alpha)$ is H\"older-continuous of any exponent $0 < s < \frac{1}{2}$.
\end{theorem}


The proof follows from spectral analysis of the transfer operator acting on the space of functions of bounded variation. 
First (section \ref{LY}), we prove a uniform bound on the essential spectral radius (Lasota-Yorke inequality).
 Then (section \ref{Kellerdistsect}), we prove that a suitable distance between the transformations $T_\alpha$ is H\"older-continuous 
in $\alpha$, and use a stability result of the spectral decomposition \cite{KellerLiverani} to prove H\"older-continuity of the invariant
 densities $\rho_\alpha$ in $L^1$-norm. Note that invariant densities are not continuous in $BV$-norm (see remark \ref{rmk}).

\medskip
The second part of the paper deals with central limit theorems.
In \cite{LuzziMarmi}, entropy is computed by approximating it with Birkhoff averages for the observable $\log |T_\alpha'|$, and
 numerical evidence is given (\cite{LuzziMarmi}, figure 3) that Birkhoff sums for different orbits distribute normally around the average. 
We first show (section \ref{CLTBVsect}) that the methods of \cite{Broise} can be used to prove a central limit theorem 
for observables of bounded variation. 

Moreover, in section \ref{CLTunbsect} we expand the class of observables we use in order to encompass unbounded observables such as the logarithm.
Indeed, we will define a new family of Banach spaces $B_{K, \delta}$, consisting of possibly unbounded functions
whose total variation grows slowly on intervals which approach zero. Such functions will be called \emph{of mild growth},
 and we prove the central limit theorem to hold in these larger spaces: 


\begin{theorem} \label{CLTunbounded}
Let $\alpha \in (0,1]$, $0 < \delta < \frac{1}{2}$, and $K$ sufficiently large. Then, for every non-constant real-valued $f \in B_{K, \delta}$ there exists $\sigma > 0$ s.t.
$$\lim_{n \rightarrow \infty} \mu_{\alpha}\left( \frac{S_n(f - \int_{I_\alpha} f d\mu_\alpha)}{\sqrt{n}} \leq v \right) = \frac{1}{\sigma \sqrt{2\pi}}\int_v^{+\infty} e^{-\frac{t^2}{2\sigma^2}} dt \qquad \forall v \in \mathbb{R}$$
\end{theorem}

As a corollary, Birkhoff sums for the observable $\log |T_\alpha'|$ distribute normally around the average value $h(T_\alpha)$.

Finally, in section \ref{last} we discuss the dependence of the standard deviation of Birkhoff averages on the parameter $\alpha$. 
More precisely, given some observable $f$ of class $C^1$, for which we proved the central limit theorem to hold, we will 
prove that the variance $\sigma^2_{\alpha, f}$ of the limit Gaussian distribution is continuous in $\alpha$. 
The result is motivated by numerical data in (\cite{CMPT}, section 2.3).

\medskip
Many different authors have studied the spectral 
properties of transfer operators of expanding maps. For instance, 
a spectral decomposition for individual expanding maps
is proved in \cite{Broise},  \cite{Rychlik} and \cite{Viana}. 
A brief historical account with references is given in \cite{KellerLiverani}.
In our case, however, it is essential to prove estimates on the spectral radius which are uniform in $\alpha$.
Since new branches of $T_\alpha$ appear as $\alpha$ moves, and $T_\alpha$ develops an indifferent fixed point
as $\alpha \to 0$, proving uniformity requires more work. 
Unfortunately, although a uniform Lasota-Yorke inequality holds, the
proof provided by \cite{LuzziMarmi} contains a bug; we shall therefore
produce a new proof in prop. \ref{LY}.
Another proof of continuity (not H\"older) of entropy is given in the very recent paper \cite{KSS}
 via a study of natural extensions. 

Let us finally remark that our functional-analytic methods only use a few properties of $T_\alpha$, hence they 
can be applied to a wider class of one-parameter families of expanding interval maps.  
For instance, they apply to the case of $(a,b)$-continued fraction transformations studied in \cite{KU}
for parameters on the critical line $b-a = 1$. 



\section{Basic properties} 

Let us start by setting up the framework needed for the rest of the paper, and establishing a basic spectral decomposition
for the transfer operator. The literature on thermodynamic formalism for interval maps is huge: the sources we mainly refer to are  
\cite{Broise}, \cite{Rychlik} and \cite{Viana}, which already make use of functions of bounded variation. 

The total variation of a function $f$ on a set $X \subseteq \mathbb{R}$ is 
$$\var_X f := \sup \sum_{i = 1}^n |f(x_i)-f(x_{i+1})|$$
where the sup is taken over all finite increasing sequences ${x_1 \leq x_2 \leq \dots \leq x_n}$ of points of $X$.
Given an interval $I$, let us denote $BV(I)$ the Banach space of complex-valued bounded variation functions of the interval $I$,
modulo equality almost everywhere. The space is endowed with the norm
$$ \Vert f \Vert_{BV(I)} := \inf \left\{ \var_{I} g + \int_{I} |g(x)| dx  \ : \ g = f \ a.e.\right\}$$
Observe that every $f \in BV(I)$ has a (not necessarily unique) representative of minimal total variation, namely such that 
$$f(x) \in [\lim_{y \to x^-} f(y), \lim_{y \to x^+} f(y) ] \qquad \forall x \in I$$
In the following, we will always choose representatives for our functions of minimal variation. Other basic properties of total variation are stated in the appendix.

\subsection{Cylinders} \label{cyl}

For each $\alpha \in  (0,1)$, the dynamical system $T_\alpha$ defined in the introduction acts on the interval 
$I_{\alpha} := [\alpha-1, \alpha]$.
Observe that there exists a partition of $I_\alpha$ in a countable number of intervals $I_j$ 
such that for every $j$ the restriction $T_{\alpha}\mid_{I_j}$ is a strictly monotone, $C^{\infty}$ function and it extends to a $C^{\infty}$ 
function on the closure of every $I_j$. The least fine of such partitions will be called $\mathcal{P}_1$, the \emph{partition associated } to 
$T_{\alpha}$. 
More specifically, $\mathcal{P}_1 = \{I_j^+\}_{j \geq j_{min}} \cup \{I_j^- \}_{j \geq 2}$ with $j_{min} = \lceil \frac{1}{\alpha} - \alpha \rceil$ where 
$$I_j^+ = \left( \frac{1}{j+\alpha}, \frac{1}{j-1+\alpha} \right) \textup{ if } j\geq j_{min}+1 \quad I_{j_{min}}^+ = \left( \frac{1}{j_{min}+\alpha}, \alpha \right)$$
$$I_j^- = \left( -\frac{1}{j-1+\alpha}, -\frac{1}{j+\alpha}  \right) \textup{ if } j\geq 3 \quad I_{2}^- = \left(\alpha-1,  \frac{1}{2+\alpha} \right)$$
Moreover, for every $n > 1$, the set
$$\{ I_{j_1}^{\epsilon_1} \cap T_{\alpha}^{-1}(I_{j_2}^{\epsilon_2}) \cap \dots \cap T^{-(n-1)}(I_{j_n}^{\epsilon_{n}})\ |\ I_{j_1}^{\epsilon_1}, \dots, I_{j_n}^{\epsilon_n} \in \mathcal{P}_1 \}$$
where $\epsilon_i\in \{+,-\}$, is a partition of $I_{\alpha}$ in a countable number of intervals such that on each of these the restriction 
of  $T_{\alpha}^n$ is monotone and $C^{\infty}$: such a partition will be denoted by $\mathcal{P}_n$ and its elements called \emph{cylinders}. 
The cylinder $I_{j_1}^{\epsilon_1} \cap T_{\alpha}^{-1}(I_{j_2}^{\epsilon_2}) \cap \dots \cap T^{-(n-1)}(I_{j_n}^{\epsilon_{n}})$ will be denoted 
either by $(I_{j_1}^{\epsilon_1}, \dots, I_{j_n}^{\epsilon_n})$ or by $((j_1, \epsilon_1), \dots, (j_n, \epsilon_n))$.
The cylinders $I_j \in \mathcal{P}_n$ such that $T_\alpha^n(I_j) = I_\alpha$ will be called \emph{full cylinders}.

Let us define the function 
$$g_{n, \alpha}(x) := \sum_{j \in \mathcal{P}_n} \frac{1}{|(T^n_{\alpha})'(x)|}\chi_{I_j}(x)$$
The following estimates, proven in the appendix, will be used throughout the paper:
 
\begin{proposition} \label{classeC}
For every $\alpha \in (0, 1)$ and for every $n \geq 1$

\begin{enumerate}
 \item 
$$ \Vert g_{n, \alpha} \Vert_\infty \leq \gamma_\alpha^n $$
where $\gamma_\alpha := \max\{ \alpha^2, (\alpha-1)^2 \}$.
\item 
$$ \sup_{j \in \mathcal{P}_n} \sup_{x \in I_j} \left| g_{n, \alpha}'(x) \right| \leq \frac{2}{1-\gamma_\alpha } $$
\item The set $\{ T^n_{\alpha}(I_j)\ |\ I_j \in \mathcal{P}_n \}$ is finite; more precisely, 
$$ \#\{ T^n_{\alpha}(I_j)\ |\ I_j \in \mathcal{P}_n \} \leq 2n + 1$$
\item
The total variation of $g_{1, \alpha}$ is universally bounded, i.e. there is a constant $C_0$ such that 
$$ \var_{I_\alpha} g_{1, \alpha} \leq C_0 < +\infty \qquad \forall \alpha \in (0,1)$$ 

\end{enumerate}
\end{proposition}

\subsection{Spectral decomposition } \label{spectral}

The \emph{transfer operator} (also known as \emph{Ruelle-Perron-Frobenius operator}) $\Phi_\alpha : L^1(I_\alpha) \rightarrow L^1(I_\alpha)$ is 
defined via the duality
$$\int_{I_\alpha} \Phi_\alpha(f) g dx = \int_{I_\alpha} f (g\circ T_\alpha) dx \quad \forall f\in L^\infty(I_\alpha)$$ 
Let us recall that the $n^{th}$ iterate of the transfer operator is given by
\begin{equation} \label{transfer}
\Phi_{\alpha}^n(f) = \sum_{j \in \mathcal{P}_n} \frac{f \circ  \sigma_j}{|(T^n_{\alpha})' \circ \sigma_j|}\chi_{T_{\alpha}^n(I_j)}
\end{equation}
where $\sigma_j : T^n_\alpha(I_j) \rightarrow I_j$ is the inverse of the restriction $T^n_\alpha\mid_{I_j} : I_j \rightarrow T^n_\alpha(I_j)$.

Even though $\Phi_\alpha$ is so far defined on $L^1$, it turns out that the transfer operator 
preserves the subspace $BV(I_\alpha)$, and indeed it has good convergence properties in $BV$-norm. 
More precisely, we can now prove the

\begin{theorem} \label{spectraldec}
Let $\Phi_{\alpha} : L^1(I_{\alpha}) \rightarrow L^1(I_{\alpha})$ be the transfer operator for the system $T_{\alpha}$, with 
$\alpha \in (0,1)$. Then one can write
$$\Phi_{\alpha} = \Pi_\alpha + \Psi_\alpha $$
where $\Pi_\alpha$ and $\Psi_\alpha$ are commuting, linear bounded operators on $BV(I_\alpha)$. Moreover, 
$\Psi_\alpha$ is a linear bounded operator on $BV(I_{\alpha})$ of spectral radius strictly less than $1$,
and $\Pi_\alpha$ is a projector onto the one-dimensional eigenspace relative to the eigenvalue $1$. It is given by
$$\Pi_\alpha(f) = \lim_{n \rightarrow \infty} \frac{1}{n} \sum_{k = 1}^{n} \Phi_{\alpha}^k(f)$$
where the convergence is in $L^1$.
\end{theorem} 

\begin{corollary}
For every $\alpha \in (0,1)$, $T_{\alpha}$ has exactly one invariant probability measure $\mu_\alpha$ which is absolutely continuous w.r.t. 
Lebesgue measure.
Its density will be denoted by $\rho_\alpha$.
\end{corollary}

\begin{proof}
Let us fix $\alpha \in (0,1)$. By proposition \ref{classeC}, we can apply (\cite{Broise}, prop. 4.1), which 
yields via Ionescu-Tulcea and Marinescu's theorem \cite{ITM} the following spectral decomposition
$$\Phi_{\alpha} = \sum_{i = 0}^p \lambda_i \Phi_i + \Psi_\alpha $$
where $|\lambda_i| = 1$, and the $\Phi_i$ are linear bounded operators on $BV(I_\alpha)$ with finite-dimensional image, and $\rho(\Psi_\alpha) < 1$.
Now, it is known (\cite{LuzziMarmi}, lemma 1) that $T_{\alpha}$ is \emph{exact} in Rohlin's sense (see \cite{Rohlin}); this implies that 
the invariant measure we have found is ergodic and mixing, which in turn implies that the only eigenvalue of $\Phi_{\alpha}$ of modulus $1$ 
is $1$ itself and that its associated eigenspace is one-dimensional (see \cite{Viana}, chap. 3).
\end{proof}

 

The spectral decomposition also immediately implies the following exponential decay of correlations: 

\begin{proposition}
For any $\alpha \in (0,1)$ there exist $C, \lambda$, $0 < \lambda < 1$ such that for every $n \in \mathbb{N}$ and for every $f_1, f_2 \in BV(I_\alpha)$ 
$$\left| \int_{I_\alpha} f_1(x) f_2(T^n_\alpha(x)) d\mu_\alpha -   \int_{I_\alpha} f_1(x) d\mu_\alpha \int_{I_\alpha} f_2(x) d\mu_\alpha \right| \leq C\lambda^n \Vert f_1 \Vert_{BV} \Vert f_2 \Vert_{L^1}$$ 
\end{proposition}

\begin{proof}
One can take any $\lambda$ s.t. $\rho(\Psi_\alpha) < \lambda <1$ and $C = 2 \Vert \rho_\alpha \Vert_{BV} \sup_{n \in \mathbb{N}} \frac{\Vert \Psi_\alpha^n \Vert_{BV}}{\lambda^n}$.  
\end{proof}

\section{Continuity of entropy} \label{LasotaYorke}

The goal of this section is to prove theorem \ref{Holdermain}, namely the H\"older-continuity of the function 
$\alpha \mapsto h(T_\alpha)$.

The first step is to prove an estimate of the essential spectral radius of the transfer operator 
acting on the space of BV functions (Lasota-Yorke inequality). 
If one can prove a bound which is uniform in $\alpha$, then the invariant densities $\rho_\alpha$
turn out to be continuous in the $L^1$-topology and their $BV$-norms are bounded. 
This method has been undertaken in \cite{LuzziMarmi}, but unfortunately their estimates prove to be too 
optimistic\footnote{The mistake in \cite{LuzziMarmi} consists in using, in eq. (12), the estimate (1) of lemma \ref{BVprop} of 
the present paper on the sets $\tilde{I}_{\xi}^{(n)}$, which are not intervals if $n >1$.}: 
the bulk of section \ref{LYsect} (prop. \ref{LY}) is another proof of this uniform Lasota-Yorke inequality.

The second step (section \ref{Kellerdistsect}) is to estimate the modulus of continuity of $h(T_\alpha)$: 
we will prove H\"older-continuous dependence of the invariant densities $\rho_\alpha$ in the $L^1$-topology, 
by using a stability result for the spectral projectors \cite{KellerLiverani}.
The theorem then follows from Rohlin's formula.


\subsection{Spectral radius estimate} \label{LYsect}

We are going to give a proof of the following uniform Lasota-Yorke inequality (in order to simplify notation, 
from now on $\var_{I_\alpha} f$ will just be denoted $\var f$):

\begin{proposition} \label{LY}
Let $\underline{\alpha} \in (0,1)$. Then there exist a neighbourhood $U$ of $\underline{\alpha}$ 
and constants $0 < \lambda < 1, C > 0, D > 0$ such that for every $\alpha \in U$
$$\var \Phi^n_\alpha(f) \leq C \lambda^n \var f + D \Vert f \Vert_{L^1} \qquad \qquad \forall n \geq 1, 
\ \forall f \in BV(I_\alpha)$$
\end{proposition}

Although several inequalities of this type are present in the literature, (i.e. in \cite{Rychlik}), 
these are generally given for individual maps. However, for the goal of this paper it is absolutely essential that 
coefficients $\lambda, C, D$ can be chosen uniformly in $\alpha$, hence one needs to take this dependence into account.
As $\alpha$ moves, even just in a neighbourhood of some fixed $\underline{\alpha}$,
 topological bifurcations are present (for instance if $\underline{\alpha}$ is a fixed point of some branch of $T_{\underline{\alpha}}$) 
hence in the formula \eqref{transfer} new boundary terms appear, requiring a very careful control. 

\begin{lemma} For each $\alpha \in (0,1)$, for each $f \in BV(I_\alpha)$
$$\var \Phi^n_{\alpha}(f) \leq  \var(f \cdot g_{n, \alpha})$$
\end{lemma}

\begin{proof}

\begin{small}

$$\var \Phi^n_{\alpha}(f) = \var \left( \sum_{j \in \mathcal{P}_n} \frac{f \circ  \sigma_j}{|(T^n_{\alpha})' \circ \sigma_j|}\chi_{T_{\alpha}^n(I_j)} \right) \leq 
\sum_{j \in \mathcal{P}_n} \var \left( \frac{f \circ  \sigma_j}{|(T^n_{\alpha})' \circ \sigma_j|}\chi_{T_{\alpha}^n(I_j)} \right) =$$
$$ = \sum_{j \in \mathcal{P}_n} \var \left( \frac{f}{|(T_\alpha^n)'|}\chi_{I_j} \right) = 
\var \left( f \sum_{j \in \mathcal{P}_n} \frac{1}{|(T_\alpha^n)'|} \chi_{I_j} \right) =
\var (fg_{n, \alpha})$$
\end{small}
 
\end{proof}

Observe that $g_{n, \alpha}$ has infinitely many jumps discontinuities (indeed it is zero on the boundary of any interval 
of the partition $\mathcal{P}_n$), but all those jumps sum up to a finite total variation. We will, however, need to prove 
the stronger statement that the total variation of $g_{n, \alpha}$ decays exponentially fast in $n$, and uniformly in $\alpha$.
The idea of the proof is to control the total variation of $g_{n, \alpha}$ by writing it as a sum of two functions, 
$h_{n, \alpha}$ and $l_{n, \alpha}$ in such a way that the total variation of $l_{n, \alpha}$ is always very small, 
and $h_{n, \alpha}$ has always a finite, controlled number, of jump discontinuities. The following lemma is the key lemma:


\begin{lemma} \label{technical}

For each $\epsilon > 0$, for each $n \geq 1$, for each $\alpha \in (0,1)$ there exist two non-negative functions 
$h_{n, \alpha}$ and $l_{n, \alpha}$ such that
$$g_{n, \alpha} = h_{n,\alpha} + l_{n, \alpha}$$
and for each $\alpha$
\begin{enumerate}
\item $\var_{I_\alpha} l_{n, \alpha} \leq 3^n C_0^{n-1} \epsilon$, \ where $C_0$ is the constant in lemma \ref{classeC};
\item $h_{n, \alpha}$ is smooth with $|h_{n, \alpha}'| \leq 1$ outside a finite set $J_{n, \alpha}$, where $h_{n, \alpha}$ has jump 
discontinuities. Moreover, for each $\alpha$ there exists a 
neighbourhood $U = (\alpha-\eta, \alpha+\eta)$ of $\alpha$ and $r>0$ such that:
\begin{itemize}
 \item[a.] 
 For each $\beta \in U$, \quad $J_{n, \beta} \subseteq B(J_{n, \alpha}, r)$
\item[b.]
For each $x \in J_{n, \alpha}$,\quad $\#|J_{n, \beta} \cap B(x, r)| \leq n+1$
\item[c.] 
For each $y \in J_{n, \beta} \cap B(x, r)$, \quad  $|x-y| \leq |\alpha - \beta|$ 
\end{itemize}
\end{enumerate}

\end{lemma}

\begin{proof}
By induction on $n$. If $n = 1$, let us note that

$$g_{1, \alpha}(x) := \left\{ 
\begin{array}{ll} x^2 & \textup{if }x \textup{ belongs to some }I_j \\
		  0 & \textup{otherwise}
 
\end{array}
\right.$$
hence we can choose $L := [-\frac{1}{K}, \frac{1}{K}]$ an interval around $0$ such that, for all $\alpha$,
$\var_L g_{1, \alpha}\leq \epsilon$ and 
define 
$$l_{1, \alpha} := g_{1, \alpha} \chi_L \qquad h_{1, \alpha} := g_{1, \alpha} \chi_{I_\alpha \setminus L}$$
1. is clearly verified. To verify 2., note that given $x \in J_{1, \alpha}, x\neq \alpha, \alpha-1$, 
for $\beta$ sufficiently close to $\alpha$, $J_{1, \beta}$ intersects a neighbourhood of $x$ in only one point.
The same happens if $x = \alpha, \alpha-1$ and $T_\alpha(x) \neq \alpha-1$.
On the other hand, if $x = \alpha$ and $T_\alpha(\alpha) = \alpha-1$, then $J_{1,\beta} \cap [\beta-\eta, \beta] = \{y, \beta\}$
contains at most two points, where $y = T_\beta^{-1}(\beta-1) \cap [\beta-\eta, \beta]$ and, since $T_\beta$ is expanding, 
$|y - \alpha| \leq |\alpha - \beta|$. The case $x = \alpha-1$, $T_\alpha(\alpha-1) = \alpha-1$ is similar.

In order to prove the inductive step, let us remark that $g_{n+1, \alpha} = g_{n, \alpha} \circ T_\alpha \cdot g_{1, \alpha}$. Hence, 
we can define 

$$\begin{array}{lll}
    h_{n+1, \alpha} & := & h_{n, \alpha} \circ T_\alpha \cdot h_{1, \alpha} \\
    l_{n+1, \alpha} & := & l_{n, \alpha} \circ T_\alpha \cdot g_{1, \alpha} + h_{n, \alpha} \circ T_\alpha \cdot l_{1, \alpha}
\end{array}$$
and check all properties are satisfied. 
First of all, we can prove by induction that 
\begin{equation} \label{variterate}
\var_{I_\alpha} h_{n, \alpha} \leq 2^{n-1} C_0^n \qquad \forall \alpha \in (0,1), \forall n \geq 1
\end{equation}
Indeed, 
\begin{small}
$$\var_{I_\alpha} h_{1, \alpha} \leq \var_{I_\alpha} g_{1, \alpha} \leq C_0$$
$$\var_{I_\alpha} h_{n+1, \alpha} = \sum_{k \in \mathcal{P}_1} \var_{\overline{I_k}} (h_{n, \alpha} \circ T_\alpha \cdot h_{1, \alpha}) \leq
\sum_{k \in \mathcal{P}_1} \var_{\overline{I}_k} (h_{n, \alpha} \circ T_\alpha) \sup_{\overline{I}_k} h_{1, \alpha} + \sup_{\overline{I}_k} (h_{n, \alpha} \circ T_\alpha) \var_{\overline{I}_k} h_{1, \alpha} \leq$$
\end{small}
and since $T_\alpha \mid_{I_k}$ is a homeomorphism 
\begin{small}
$$ \leq \var_{I_\alpha} h_{n, \alpha} \sum_{k \in \mathcal{P}_1} \sup_{\overline{I}_k} h_{1, \alpha} + \sup_{I_\alpha} h_{n, \alpha} \sum_{k \in \mathcal{P}_1} \var_{\overline{I}_k} h_{1, \alpha}
\leq 2 \var_{I_\alpha} h_{n, \alpha} \var_{I_\alpha} h_{1, \alpha}\leq 2 \cdot 2^{n-1}C_0^{n} \cdot C_0$$ 
\end{small}
where in the penultimate inequality we used the fact that $\sup_I f \leq \var_I f$ if $f(x) = 0$ for some $x \in I$. 

Let us now check 1.: similarly as before,
$$\var_{I_\alpha} l_{n+1, \alpha} =
\var_{I_\alpha}  (l_{n, \alpha} \circ T_\alpha \cdot g_{1, \alpha} + h_{n, \alpha} \circ T_\alpha \cdot l_{1, \alpha}) \leq
 2 \var_{I_\alpha} l_{n, \alpha} \var_{I_\alpha} g_{1, \alpha} + 2 \var_{I_\alpha} h_{n, \alpha} \var_{I_\alpha} l_{1, \alpha} \leq$$
and by inductive hypothesis and \eqref{variterate}
$$\leq 2 \cdot 3^n C_0^{n-1} \epsilon \cdot C_0 + 2 \cdot 2^{n-1} C_0^n \cdot \epsilon \leq 3^{n+1} C_0^n \epsilon$$

Since $h_{1, \alpha}$ is nonzero only on finitely many branches of $T_\alpha$, then $h_{n+1, \alpha}$ has only finitely many jump 
discontinuities. Now, if $x$ is a jump discontinuity for $h_{n, \alpha} \circ T_\alpha$ and not for $h_{1, \alpha}$, 
then $T_\beta$ is an expanding local homeomorphism at $x$ for all $\beta$ in a neighbourhood of $\alpha$, hence a., b. and c. follow.
Let now $x \neq \alpha, \alpha-1$ be on the boundary of some cylinder, i.e. $T_\alpha(x) = \alpha-1$.
Then by inductive hypothesis c., if $\beta > \alpha$ is sufficiently close to $\alpha$ and $\eta$ is sufficiently small, then
$$J_{n, \beta} \cap [\beta-1, \beta -1 + \eta] = \{\beta-1\}$$
hence $$J_{n+1, \beta} \cap B(x, r) = T_{\beta}^{-1}(J_{n, \beta}\cap [\beta-\eta, \beta]) \cap B(x, r)$$
and b. follows. c. follows from the fact that $T_\beta$ is expanding. If $\beta < \alpha$, similarly the claims follow because
$$J_{n+1, \beta} \cap B(x, r) = T_{\beta}^{-1}(J_{n, \beta}\cap [\beta-1, \beta-1+\eta]) \cap B(x, r)$$
If $x = \alpha$, then for $\beta$ sufficiently close to $\alpha$, 
$$J_{n+1, \beta} \cap B(x, r) \subseteq (T_\beta^{-1}(J_{n, \beta}) \cup \{\beta\}) \cap B(x, r)$$ 
has cardinality at most $n+2$, and c. follows because $T_\alpha$ is expanding. 
The case $x = \alpha-1$ is analogous.

\end{proof}

\begin{lemma} \label{goodpartition}
Let $\alpha \in (0,1)$, $n \geq 1$ and $\epsilon > 0$. Then there exist $\eta > 0$, $C > 0$ and a finite partition
 of $[\alpha-1-\eta, \alpha +\eta]$ in closed intervals $L_1, \dots L_r$ such that for each $\beta \in (\alpha - \eta, \alpha + \eta)$
and each $i \in {1, \dots, r}$ the following holds:
\begin{itemize}
 \item $ 0 < C \leq m(L_{i, \beta}) \leq \epsilon$
\item $\var_{L_{i, \beta}} g_{n, \beta} \leq 2(n+1) \Vert g_{n, \beta} \Vert_\infty + 2\epsilon$
\end{itemize}
where $L_{i, \beta} := L_i \cap [\beta-1, \beta]$.
\end{lemma}

\begin{proof}
Given $\alpha, n, \epsilon$, choose $L_1, \dots, L_r$ in such a way that $m(L_i) \leq \epsilon$, each element of $J_{n, \alpha}$ 
lies in the interior of some $L_i$ and no two such elements lie in the same $L_i$. Moreover, 
set $\epsilon_1 := \epsilon/(3^nC_0^{n-1})$ and, for each $\beta$ sufficiently close to $\alpha$, choose a decomposition
$g_{n, \beta} = h_{n, \beta} + l_{n, \beta}$ as in lemma \ref{technical} relative to $\epsilon_1$.
  
$$\var_{L_{i, \beta}} h_{n, \beta} \leq \int_{L_{i, \beta} \setminus J_{n, \beta}} h'_{n, \beta}(x) dx + \sum_{x \in L_{i, \beta} \cap J_{n, \beta}} \lim_{y \to x^-} h_{n, \beta}(y) + \lim_{y \to x^+} h_{n, \beta}(y) \leq $$
$$\leq m(L_{i, \beta}) + 2 \#\{L_{i, \beta} \cap J_{n, \beta}\} \Vert h_{n, \beta} \Vert_\infty \leq \epsilon + 2(n+1) \Vert h_{n, \beta} \Vert_\infty$$
hence $\var_{L_{i, \beta}} g_{n, \beta} \leq \var_{L_{i, \beta}} h_{n, \beta} + l_{n, \beta} \leq 2 \epsilon + 2(n+1)\Vert g_{n, \beta} \Vert_\infty$. 
\end{proof}

\noindent \textit{Proof of proposition \ref{LY}. } 
Consider the partition $L_1, \dots, L_r$ given by lemma \ref{goodpartition}. Then
$$\var (f \cdot g_{n, \alpha} ) = \sum_{i = 1}^r \var_{L_i} (f g_{n, \alpha}) \leq \sum_{i=1}^r \var_{L_i} f \sup_{L_i} g_{n, \alpha} +
 \var_{L_i} g_{n, \alpha} \sup_{L_i} f  \leq $$
$$ \leq \sum_{i=1}^r \Vert g_{n, \alpha}   \Vert_\infty  \var_{L_i} f + \var_{L_i} g_{n, \alpha} \left( \frac{1}{m(L_{i, \alpha})} \int_{L_i} f(x) dx + \var_{L_i} f \right) \leq $$
$$\leq [(2n+3) \Vert g_{n, \alpha} \Vert_\infty + 2 \epsilon] \var_{I_\alpha} f + \frac{(2n+2) \Vert g_{n,\alpha} \Vert_\infty + 2 \epsilon}{C} \int_{I_\alpha} f(x) dx $$
Now, since $\Vert g_{n, \alpha} \Vert_\infty \leq \gamma_\alpha^n$ decays exponentially,
 we can choose $n$ large enough so that $\lambda:= (2n+4) \gamma_\alpha^n < 1$, and we can also choose $2 \epsilon \leq \gamma_\alpha^n$, 
hence we get that for some constant $D >0$, for each $\alpha \in (\underline{\alpha} - \eta, \underline{\alpha} + \eta)$, 
\begin{equation} \label{finalLY}
\var \Phi_\alpha^n(f) \leq \lambda \var f + D \Vert f \Vert_1
\end{equation}
and by iteration and euclidean division (see e.g. \cite{Rychlik}, lemma 7 and prop. 1) the claim is proven. 
\qed

\subsection{Stability of spectral decomposition} \label{Kellerdistsect}

The next step to prove H\"older-continuity is proving the continuous dependence of invariant densities $\rho_\alpha$ in $L^1$-norm. 
In order to guarantee the stability of spectral projectors of the transfer operator, we will use the following theorem of Keller and Liverani \cite{KellerLiverani}:

\begin{theorem} \label{KellerLiv}
 Let $P_{\epsilon}$ be a family of bounded linear operators on a Banach space $(B, \Vert \cdot \Vert)$ which is also equipped with a second 
norm $\vert \cdot \vert$ such that $\vert \cdot \vert \leq \Vert \cdot \Vert$. Let us assume that the following conditions hold: 
\begin{enumerate}
 \item $\exists C_1, M > 0$ s.t. for all $\epsilon \geq 0$
$$|P_{\epsilon}^n| \leq C_1 M^n \qquad \forall n \in \mathbb{N}$$
\item $\exists C_2, C_3 > 0$ and $\lambda \in (0,1)$, $\lambda < M$, such that for all $\epsilon \geq 0$
$$\Vert P_\epsilon^n f \Vert \leq C_2 \lambda^n \Vert f \Vert + C_3 M^n \vert f \vert \quad \forall n \in \mathbb{N} \quad \forall f \in B$$
\item if $z \in \sigma(P_{\epsilon}), |z| > \lambda$, then $z$ is not in the residual spectrum of $P_\epsilon$
\item There is a monotone continuous function $\tau : [0, \infty) \rightarrow [0, \infty)$ such that $\tau(\epsilon) > 0$ if $\epsilon > 0$ and 
$$ \vert \! \vert \! \vert P_0 - P_\epsilon \vert \! \vert \! \vert \leq \tau(\epsilon) \rightarrow 0 \quad as \ \epsilon \rightarrow 0$$
where the norm $| \! | \! | \cdot |\!|\!|$ is defined as 
$$| \! | \! | Q | \! | \! | := \sup_{\Vert f \Vert \leq 1} |Qf|$$
\end{enumerate}

Let us now fix $\delta > 0$ and $r \in (\lambda, M)$ and define
$$V_{\delta, r} := \{ z \in \mathbb{C} : |z| \leq r \textrm{ or dist }(z, \sigma(P_0)) \leq \delta \}$$
and $\eta := \frac{\log(\lambda/r)}{\log(\lambda/M)}$. Then there exist $H, K > 0$ such that if $\tau(\epsilon) \leq H$ then $\sigma(P_{\epsilon}) \subseteq V_{\delta, r}$ and 
$$\vert \! \vert \! \vert (z-P_{\epsilon})^{-1} - (z-P_0)^{-1} \vert \! \vert \! \vert \leq K \tau(\epsilon)^{\eta} \qquad \forall z \notin V_{\delta, r}$$
\end{theorem}
\noindent In our context, the norm $|\cdot|$ will be the $L^1$ norm and $\Vert \cdot \Vert$ will be the $BV$ norm. 
Our goal is to apply this result to the family
$\{\Phi_\alpha\}_{\alpha \in U}$ where $U$ is a suitable neighbourhood of a given $\underline{\alpha} \in (0,1)$.

Hypothesis (1) is trivial since transfer operators have unit $L^1$-norm, and (2) is precisely  
proposition \ref{LY}.
In the context of one-dimensional piecewise expanding maps, (3) is an immediate corollary of (2):

\begin{lemma} \label{henn}
For every $\underline{\alpha} \in (0,1)$ there exists $\epsilon > 0$ such that for
$|\alpha - \underline{\alpha}| < \epsilon$, $$\rho_{ess}(\Phi_{\alpha}) \leq \lambda$$
where $\lambda$ is the same as in proposition \ref{LY} and therefore condition (3) holds.
\end{lemma}

\begin{proof}
By a result of Hennion \cite{Hennion}, the uniform Lasota-Yorke inequality plus the fact that the injection 
$BV(I) \rightarrow L^1(I)$ is compact implies the estimate on the essential spectral radius; therefore the elements 
of the spectrum of modulus bigger than $\lambda$ are eigenvalues with finite multiplicity and cannot belong to the residual spectrum. 
\end{proof}

To prove condition (4) it is necessary to estimate the distance between the $\Phi_{\alpha}$ as $\alpha$ varies in a neighbourhood of a fixed $\underline{\alpha}$; by a result of Keller \cite{Keller} the distance between the 
transfer operators is related to the following distance between the transformations:

\begin{definition}
Let $T_1, T_2 : I \rightarrow I$ two maps of the interval $I$. We define the \emph{Keller distance} between $T_1$ and $T_2$ as
$$d(T_1, T_2) := \inf \{ \kappa > 0 \ | \ \exists A \subset I \textup{ measurable with }m(A) > 1 - \kappa, $$
$$  \exists \sigma : I \rightarrow I \ \textup{diffeo s.t. } T_1\mid_A = T_2\circ \sigma\mid_A, \ \sup_{x \in I}|\sigma(x) -x| < \kappa,\ \sup_{x \in I}\left|\frac{1}{\sigma'(x)}-1\right| < \kappa \}$$
\end{definition}
 
\begin{lemma}[\cite{Keller}, lemma 13] \label{Kellerlemma}
If $P_1$ and $P_2$ are the transfer operators associated to the interval maps $T_1$ and $T_2$, then $\vert \! \vert \! \vert P_1 - P_2 \vert \! \vert \! \vert \leq 12 d(T_1, T_2)$
where $d$ is the Keller distance.
\end{lemma}

We verify now that this convergence result applies to our case of $\alpha$-continued fractions. In order to do so, it is necessary to translate the maps in such a way that they are all defined on the same interval, which will be $[0,1]$ in our case. We therefore consider the maps $\tilde{T}_{\alpha} : [0,1] \rightarrow [0,1]$
$$ \tilde{T}_{\alpha}(x) = T_{\alpha}(x+\alpha-1) + 1-\alpha $$ 
The relative invariant densities will be 
$$\tilde{\rho}_\alpha(x) = \rho_\alpha(x+ \alpha-1)$$

\begin{lemma} \label{Kellerdist}
Fix $\underline{\alpha} \in (0,1)$. Then there exists a neighbourhood $U$ of $\underline{\alpha}$ and 
a positive constant $C$ such that, for $\alpha, \beta \in U$, we have 
$$ d(\tilde{T}_{\alpha}, \tilde{T}_{\beta}) \leq C |\alpha -\beta|^{1/2}$$
\end{lemma}

\begin{proof}
Having fixed $\alpha, \beta$, let us define
$$y(x) := \frac{x+\alpha-1}{1+(\beta-\alpha)|x+\alpha-1|} + 1 - \beta$$
It is immediate to verify that $\tilde{T}_{\alpha}(x) = \tilde{T}_{\beta}(y(x)) \  \forall x \in [0,1]$ and $y'(x) = \frac{1}{(1+(\beta-\alpha)|x+\alpha-1|)^2}$
We also have 
$$\sup_{x \in [0,1]} |y(x) - x| = |y(1) - 1|$$
because when $|\alpha -\beta|$ is small we have that, for $\alpha > \beta$,  $y(x) - x$ has positive derivative and $y(0) > 0$, while, for $\alpha < \beta$, $y(x) -x$ has negative derivative and $y(0) < 0$.
Thus, for $|\alpha - \beta|$ sufficiently small, 
$$ \sup_{x \in [0,1]} |y(x) - x| = |\alpha - \beta|\left| \frac{1+ \alpha \beta}{1+ (\beta-\alpha)\alpha}\right| \leq 2 |\alpha - \beta|$$
$$ \sup_{x \in [0,1]} \left| \frac{1}{y'(x)}-1 \right| = |\beta-\alpha| \sup_{x \in [0,1]} \left|2 |x+\alpha-1| + (\beta-\alpha)|x+\alpha-1|^2\right| \leq 3 |\alpha - \beta|$$
In order to compute the Keller distance we need to find a diffeomorphism $\sigma$ of the interval such that $\tilde{T}_{\alpha} = \tilde{T}_{\beta} \circ \sigma$ on a set of large measure; the $y$ defined so far is not a diffeo, so it is necessary to modify it a bit at the endpoints and we will do it by introducing two little linear bridges. Let $\delta$ be such that $\delta^2 = \sup_{x \in [0,1]}|y(x) - x | \leq 2 |\alpha -\beta|$;  we can define
$$ \sigma(x) = \left\{ \begin{array}{ll} \frac{y(\delta)}{\delta} x &  \textrm{for }x \leq \delta \\ y(x) & \textrm{for }\delta \leq x \leq 1-\delta \\
\frac{1-y(1-\delta)}{\delta}(x-1+\delta) + y(1-\delta) & \textrm{for }x \geq 1-\delta \end{array} \right. $$
For the sup norm we have 
$$\sup_{x \in [0,1]} |\sigma(x) - x| \leq \max \left\{ |y(\delta) - \delta|, \sup_{x \in [\delta, 1-\delta]}  |y(x)-x|,  |y(1-\delta) -1+\delta| \right\} \leq $$
$$\leq \sup_{x \in[0,1]} |y(x) - x| \leq 2 |\beta - \alpha|$$
Since $|y(\delta)| \geq \delta - |y(\delta) - \delta| \geq \delta - \delta^2$, one gets
$ \sup_{x \in[0,\delta]}\left| \frac{1}{\sigma'(x)}-1 \right| \leq \frac{\delta}{1-\delta} $ and 
$$\sup_{x \in [0,1]} \left| \frac{1}{\sigma'(x)}-1\right| \leq \max \left\{ 3|\alpha-\beta|, \frac{\delta}{1-\delta} \right\} \leq C |\alpha - \beta|^{1/2} $$


Now, $\sigma$ is a homeomorphism of $[0,1]$ with well defined, non-zero derivative except for the points $x = \delta, 1-\delta$. 
Hence one can construct smooth approximations $\sigma_n$ of $\sigma$ which coincide with it except on 
$[\delta-\frac{1}{2^n}, \delta+\frac{1}{2^n}] \cup [1-\delta -\frac{1}{2^n}, 1 - \delta+\frac{1}{2^n}]$ and such that previous estimates still 
hold. These $\sigma_n$ will be diffeomorphisms of the interval s.t.
$\tilde{T}_\alpha(x) = \tilde{T}_{\beta}(\sigma_n(x))$ for $x \in [\delta+ \frac{1}{2^n}, 1-\delta -\frac{1}{2^n}]$. 
Since $\sup m([\delta+ \frac{1}{2^n}, 1-\delta -\frac{1}{2^n}]) = 1 - 2\delta \geq 1 - 2 |\alpha - \beta|^{1/2}$, then
the claim is proven.
\end{proof}

\subsection{H\"older-continuity of entropy} \label{Holdercont}

By using the perturbation theory developed so far, we complete the proof that the function $\alpha \mapsto h(T_\alpha)$ is 
locally H\"older-continuous. Note that the uniform Lasota-Yorke inequality proven in section \ref{LYsect} would already imply
 continuity by the methods in \cite{LuzziMarmi}, while here we get a quantitative bound on the continuity module.

\begin{proposition}
Let $\delta > 0$, and $0 < s < \frac{1}{2}$. Then there exists a constant $C > 0$ such that 
$$| h(T_{\alpha})-h(T_{\beta}) | \leq C |\alpha - \beta|^s \qquad \forall \alpha, \beta \in [\delta, 1]$$ 
\end{proposition}

\begin{proof} 
Let us fix $\eta \in (0,1)$, and choose $r$ such that $\eta = \frac{\log(\lambda/r)}{\log(\lambda)}$. 
By thm \ref{KellerLiv} applied to the family $\Phi_\alpha$, (using proposition \ref{LY}, lemma \ref{henn} and lemma \ref{Kellerdist} as hypotheses),
for each $\alpha \in (0,1)$ there exist $\epsilon, C_1 > 0$ such that 
$$\vert \! \vert \! \vert \Pi_{\alpha} - \Pi_{\beta} \vert \! \vert \! \vert \leq C_1 |\alpha-\beta|^{\eta/2} \qquad \forall \beta \in (\alpha-\epsilon, \alpha + \epsilon)$$

Now, in theorem \ref{KellerLiv} the bounds ($H, K$) depend only on the constants $C_1$, $C_2$, $C_3$, $\lambda$, $M$, 
and in proposition \ref{LY} and lemma \ref{Kellerdist} these constants are locally uniformly bounded in $\alpha$, hence 
the following stronger statement is true: for each $\underline{\alpha} \in (0,1)$ there is $C_1 >0$ and some neighbourhood $U$ of $\underline{\alpha}$ such that 
$$\vert \! \vert \! \vert \Pi_{\alpha} - \Pi_{\beta} \vert \! \vert \! \vert \leq C_1 |\alpha-\beta|^{\eta/2} \qquad \forall \alpha, \beta \in U$$
Since $\tilde{\rho_{\alpha}} = \Pi_{\alpha}(1)$, the previous equation implies  
$$\Vert \tilde{\rho}_{\alpha} - \tilde{\rho}_{\beta} \Vert_{L^1} = O(\left|\alpha-\beta\right|^{\frac{\eta}{2}})$$ 
By prop. \ref{LY}, $\Vert \tilde{\rho}_{\alpha} \Vert_{BV}$ is locally bounded, hence so is $\Vert \tilde{\rho}_\alpha \Vert_{\infty}$ and for any $p > 1$
$$\Vert \tilde{\rho}_{\alpha} - \tilde{\rho}_{\beta} \Vert_{L^p} = O(\left|\alpha-\beta\right|^{\frac{\eta}{2p}})$$ 
By Rohlin's formula, $h(T_\alpha) = -2 \int_0^1 \log|y+\alpha-1| \tilde{\rho}_\alpha(y)dy$, thus 
$$|h(T_{\alpha}) - h(T_{\beta})| \leq 2\int_{0}^1 \left|\log |y+\alpha-1| \tilde{\rho}_{\alpha}(y) - \log |y + \beta -1| \tilde{\rho}_{\beta}(y)\right|dy \leq $$
by separating the product and applying H\"older's inequality, for any $p > 1$
\begin{small}
$$ \leq 2 \Vert \tilde{\rho}_\alpha \Vert_{\infty} \Vert \log |y+\alpha-1| - \log|y+\beta-1|\Vert_{L^1} + \Vert 2 \log(y+\beta -1)
 \Vert_{L^{p/p-1}} \Vert \tilde{\rho}_{\alpha} - \tilde{\rho}_{\beta} \Vert_{L^p} $$
\end{small}
Now, basic calculus shows $\Vert \log |y+\alpha-1| - \log|y+\beta-1|\Vert_{L^1} = O(-|\alpha- \beta| \log |\alpha-\beta|)$
and $\Vert 2 \log(y + \alpha - 1) \Vert_{L^{p/p-1}}$ is bounded independently of $\alpha$. 
Since this is true $\forall \eta < 1$ and $\forall p > 1$, the claim follows.
\end{proof}

\begin{remark} \label{rmk}
One has to be careful with the norm he uses to get the convergence, because while $L^1$-convergence of the densities is assured by uniform Lasota-Yorke, the invariant densities in general DO NOT converge to each other in $BV$-norm. For example we have for $\alpha \geq \frac{\sqrt{5}-1}{2}$
$$ \rho_\alpha(x) = \frac{1}{\log(1+\alpha)} \left(\chi_{[0, \frac{1-\alpha^2}{\alpha}]}(x)\frac{1}{x+2} + \chi_{(\frac{1-\alpha^2}{\alpha}, 1]}(x) \frac{1}{x+1} \right) $$
so
\begin{small}
$$\var_{[0,1]}(\tilde{\rho}_\alpha - \tilde{\rho}_{\underline{\alpha}}) \geq \left| \lim_{x \rightarrow \left(\frac{1-\alpha^2}{\alpha}\right)^{-}} (\rho_\alpha - \rho_{\underline{\alpha}}) - \lim_{x \rightarrow \left(\frac{1-\alpha^2}{\alpha}\right)^{+}} (\rho_\alpha - \rho_{\underline{\alpha}}) \right|$$
\end{small}
which does not converge to $0$ as $\alpha \rightarrow \underline{\alpha}$.

\end{remark}

\section{Central limit theorems} \label{CLTsect}

The goal of this section is to prove a central limit theorem (CLT) for the systems $T_{\alpha}$.
Given an observable $f : I_\alpha \rightarrow \mathbb{R}$, we denote by $S_nf$ the \emph{Birkhoff sum}
$$S_n f = \sum_{j = 0}^{n-1} f \circ T_{\alpha}^j$$
The function $x \mapsto \frac{S_n f(x)}{n}$ is called \emph{Birkhoff average} and it can be seen as a random variable 
on the space $I_\alpha = [\alpha-1, \alpha]$ endowed with the measure $\mu_{\alpha}$. 
By ergodicity, this random  variable converges a.e. to a constant. 
Our goal is to prove that the difference from such limit value converges in law to a Gaussian distribution. 

Heuristically, this means the sequence of observables $\{f \circ T_\alpha^n\}$ 
(which can be seen as  identically distributed random variables on $I_\alpha$) behave as if they were independent, i.e. the system has 
little memory of its past. A convergence property of this type is also useful to confirm numerical data, since it implies the variance of 
Birkhoff averages up to the $n^{th}$ iterate decays as $\frac{1}{\sqrt{n}}$, hence one can get a good approximation of the limit 
value by computing Birkhoff averages up to a suitable finite time $n$ (see \cite{LuzziMarmi}).  

First (subsection \ref{CLTBVsect}), we will prove CLT for observables of bounded variation. 
A particularly important observable is $\log|T_\alpha'|$, because by Rohlin's formula its expectation is the metric entropy.
Such observable, however, is not of bounded variation: in subsection \ref{CLTunbsect}, we will enlarge the class of observables we work with in order to encompass certain 
unbounded functions, including $\log|T_\alpha'|$. In order to do so, we need to define \emph{ad hoc} Banach spaces. 

\subsection{CLT for functions of bounded variation} \label{CLTBVsect}


\begin{theorem} \label{CLT}
Let $\alpha \in (0,1]$ and $f$ be a real-valued nonconstant element of $BV(I_{\alpha})$.  There exists $\sigma > 0$ such that the random variable
$\frac{S_n (f -\int f d\mu_{\alpha})}{\sigma \sqrt{n}}$ converges in law to a Gaussian $\mathcal{N}(0,1)$, i.e. for every $v \in \mathbb{R}$ we have 
$$ \lim_{n \rightarrow \infty} \mu_{\alpha}\left(\frac{S_n f - n \int_I f d\mu_{\alpha}}{\sigma \sqrt{n}} \leq v \right) = \frac{1}{\sqrt{2 \pi}}\int_{-\infty}^v e^{-x^2/2}dx$$
\end{theorem}



The proof of theorem \ref{CLT} follows a method developed by A. Broise \cite{Broise}.

\medskip
\noindent \textbf{Perturbations of $\Phi_{\alpha}$}

Given $f \in BV(I_\alpha)$ with real values and given $\theta \in \mathbb{C}$, let us define the operator $\Phi_f(\theta): BV(I_\alpha) \rightarrow BV(I_\alpha)$ with
$$\Phi_f(\theta)(g) = \Phi(\textup{exp}(\theta f)g)$$
For $f$ fixed, this family of operators has the property that $\Phi_f(0) = \Phi$ and the function $\theta \rightarrow \Phi_f(\theta)$ is analytic; the interest in this kind of perturbations resides in the identity
$$\Phi_f^n(\theta)(g) = \Phi^n(\textup{exp}(\theta S_n f)g)  \qquad \textup{with }S_nf = \sum_{k = 0}^{n-1} f \circ T_{\alpha}^k$$ 
Since in our case all eigenvalues of modulus $1$ are simple, the spectral decomposition transfers to the perturbed operator:

$$ \Phi^n_f(\theta)(g) = \lambda^n_0(\theta)\Phi_0(\theta)(g) + \Psi^n_f(\theta)(g) $$
where the functions $\theta \mapsto \Phi_0(\theta)$, $\theta \mapsto \lambda_0(\theta)$ and $\theta \mapsto \Psi_f(\theta)$ are analytic in a neighbourhood of $\theta = 0$. Moreover, 
$\rho(\Psi_f(\theta)) \leq \frac{2 + \rho(\Psi)}{3} \leq |\lambda_0(\theta)|$.

\vskip 0.5 cm

\noindent Let us now consider the variance of $S_n f$:

\begin{proposition}[\cite{Broise}, thm. 6.1] \label{variance}
Let $\alpha \in (0,1]$ and $f$ be a real-valued element of $BV(I_{\alpha})$.
Then the sequence 
$$M_n = \int_{I_\alpha}\left(\frac{S_n f- n \int f d\mu_{\alpha}}{\sqrt{n}}\right)^2 d\mu_{\alpha}$$ converges to a real nonnegative value, which will be denoted by $\sigma^2$. Moreover, $\sigma^2 = 0$ if and only if there exists 
$u \in L^2(\mu_{\alpha})$ such that $u\rho_{\alpha} \in BV(I_{\alpha})$ and
\begin{equation} \label{cohom} 
f - \int_{I_\alpha} f d\mu_{\alpha} = u - u \circ T_{\alpha} 
\end{equation}
\end{proposition}

Now, if $\sigma > 0$, the method of (\cite{Broise}, chap. 6) yields the central limit theorem. The main steps in the argument are: 

\begin{enumerate}
\item
$\lambda_0'(0) = \int_{I_{\alpha}} f d\mu_{\alpha}$
\item
If $\int_{I_\alpha} f d\mu_{\alpha} = 0$, then $\lambda_0''(0) = \sigma^2$ 
\item
If $\int_{I_\alpha} f d\mu_{\alpha} = 0$, then $\lim_{n \rightarrow \infty} \int_{I_\alpha} \Phi_f^n(\frac{it}{\sigma \sqrt{n}})(\rho_\alpha)dm = \textup{exp}(-\frac{t^2}{2})$
\end{enumerate}
CLT then follows by L\'evy's continuity theorem, LHS in previous equation being the characteristic function of the random variable $\frac{S_n(f  - \int_{I_\alpha} f d\mu_{\alpha})}{\sigma \sqrt{n}}$.

In order to prove the CLT for a given observable we are now left with checking that equation (\ref{cohom}) has no solutions. The following proposition completes the proof of theorem \ref{CLT}.

\begin{proposition}
For every real-valued nonconstant $f \in BV(I_\alpha)$, equation (\ref{cohom}) has no solutions.
\end{proposition}

\begin{proof}
By proposition \ref{classeC}, $T_\alpha$ satisfies the hypotheses of a theorem of Zweim\"uller \cite{Zweimuller}, which asserts that there 
exists $C_\alpha > 0$ such that  $\rho_\alpha \geq C_\alpha$ on $\{\rho_\alpha \neq 0 \}$. Hence, the function
 $\frac{1}{\rho_{\alpha}} \chi_{\{\rho_{\alpha} \neq 0\}}$ belongs to $BV(I_\alpha)$, so if it exists $u$ such that 
$f\rho_\alpha -(\int_{I_\alpha} f d\mu_\alpha) \rho_\alpha = u\rho_\alpha - u \circ T_\alpha \cdot \rho_\alpha$ in $BV(I_\alpha)$, then
 we can multiply by $\frac{1}{\rho_\alpha} \chi_{\{\rho_\alpha \neq 0\}}$ and get $f - \int_{I_\alpha} f d\mu_\alpha = u - u \circ T_\alpha$ in
 $BV(I_\alpha)$, with $u$ in $BV(I_\alpha)$ because $u\rho_\alpha \in BV(I_\alpha)$; by knowing that $f \in BV(I_\alpha)$ we get 
$u \circ T_\alpha \in BV(I_\alpha)$. For each cylinder $I_j \in \mathcal{P}_1$, since $T_\alpha\mid_{I_j}: I_j \rightarrow I_\alpha$ is a homeomorphism, 
$$\var_{I_j} (u \circ T_\alpha) = \var_{T_\alpha(I_j)} u$$
hence 
$$ \var_{I_\alpha} (u \circ T_\alpha) \geq \sum_{\substack{I_j  \in \mathcal{P}_1 \\ I_j \textup{ full}}} \var_{I_j} (u \circ T_\alpha) 
\geq \sum_{\substack{I_j \in \mathcal{P}_1 \\ I_j \textup{ full}}} \var_{T_\alpha(I_j)} u = \sum_{\substack{I_j \in \mathcal{P}_1 \\ I_j \textup{ full}}} \var_{(\alpha - 1, \alpha)} u$$
and, since the set of $j$ s.t. $I_j$ is full is infinite, $u \circ T_\alpha$ has a representative 
with bounded variation only if $\var_{(\alpha-1, \alpha)} u = 0$, i.e. $u$ is constant a.e.
\end{proof}

\subsection{CLT for unbounded observables} \label{CLTunbsect}

In order to prove a central limit theorem for the entropy $h(T_\alpha)$ 
one has to consider the observable $x \mapsto \log|T_\alpha'(x)| = - 2 \log |x|$, which is not of bounded variation on intervals containing $0$.
 Therefore, one has to enlarge the space of functions to work with so that it contains such observable, and use some norm which still allows to 
bound the essential spectral radius of the transfer operator. Such technique will be developed in this section.

The strategy is to use the Ionescu-Tulcea and Marinescu theorem to get a spectral decomposition of the transfer operator, 
as we did in section \ref{spectral}. This theorem requires a pair of Banach spaces contained in each other such that the operator 
preserves both. Traditionally, this is achieved by considering the pair $BV(I) \subset L^1(I)$. In our case, we will replace the space 
of functions of bounded variation with newly-defined, larger spaces $B_{K, \delta} \subseteq L^1$, which allow for functions with a 
mild singularity in $0$.

\subsubsection{A new family of Banach spaces}

Fix $\alpha \in (0, 1]$. 
Given a positive integer $K$ and some $0 < \delta < 1$, let us define the $K, \delta$\emph{-norm} of a function 
$f : I_\alpha \rightarrow \mathbb{C}$ as
$$\Vert f \Vert_{K, \delta} := \sup_{k \geq K} \left( k^{-\delta} \var_{L_k} f + \int_{L_k} |f(x)| dx \right)$$
where the $L_k$ are a sequence of increasing subintervals of $I_\alpha$, namely 

$$L_k^+ := \overline{\bigcup_{j \geq k} I^+_j} = \left[ \frac{1}{k+\alpha}, \alpha \right] \quad L_k^- := \overline{\bigcup_{j \geq k} I^-_j} = \left[\alpha-1,  
-\frac{1}{k+\alpha}\right]$$
and $L_k := L_k^+ \cup L_k^-$, with $\var_{L_k} f := \var_{L_k^+} f + \var_{L_k^-} f$. 
Let us now define the space $B_{K, \delta}$ of \emph{functions of mild growth} as 
$$B_{K, \delta} := \{ f \in L^1 \ : \ f \textup{ has a version }g \textup{ with }\Vert g \Vert_{K, \delta} < \infty \}$$
Let us now establish some basic properties of these spaces. First of all, they are Banach spaces:

\begin{proposition} For every $K \in \mathbb{N}$, $0 < \delta < 1$, the space
$B_{K, \delta}$ endowed with the norm 
$$\Vert f \Vert_{K, \delta} := \inf \{ \Vert g \Vert_{K, \delta}, \ g = f \textup{ a.e.} \}$$
 is a Banach space.
\end{proposition}

\begin{proof}
This is obviously a normed vector space. Let us prove completeness. If $\{ f_n \}$ is a Cauchy sequence, then there exists for every $k$ a function $\overline{f}_k$ such that $f_n\mid_{L_k} \rightarrow \overline{f}_k$ in $BV(L_k)$-norm for $n \rightarrow \infty$. Also, by restricting $f_n\mid_{L_{k+1}} \rightarrow \overline{f}_{k+1}$ to $L_k$ one can conclude $\overline{f}_{k+1}\mid_{L_k} = \overline{f}_k$, hence one can define $f : [\alpha-1, \alpha]\setminus\{0\} \rightarrow \mathbb{C}$ s.t. $f\mid_{L_k} = \overline{f}_k$. Now, $\forall \epsilon > 0$ $\exists N$ $\forall m,n\geq N$ $\forall k \geq K$
$$\Vert f_m - f_n \Vert_{L^1(L_k)} + k^{-\delta} \var_{L_k}(f_m - f_n) \leq \epsilon$$
and by taking the limit for $n \rightarrow \infty$ one has $\Vert f_m - f \Vert_{K, \delta} \leq \epsilon$.
\end{proof}
Note that $\Vert f \Vert_{L^1(I_\alpha)} \leq \Vert f \Vert_{K, \delta}$, and $B_{K, \delta}$ is a $BV$-module, i.e. 
$$f \in B_{K, \delta}, g \in BV \Rightarrow fg \in B_{K, \delta}$$
Another useful property of these spaces is the following:

\begin{lemma} \label{powergrowth}
For $K > \max \left\{\frac{1}{\alpha} ,\frac{1}{1-\alpha} \right\}
$ $(K \geq 1$ if $\alpha = 1)$, $0 < \delta < 1$, $\exists A >0$ s.t. $\forall f\in B_{K. \delta}$
$$|f(x)| \leq  \frac{A}{|x|^\delta} \Vert f \Vert_{K, \delta} \qquad \forall x \in [\alpha-1, \alpha]\setminus\{ 0\}$$ 
\end{lemma}

\begin{proof}
For $f \in B_{K, \delta}$, $x \in L_k^+\setminus L_{k-1}^+$
$$|f(x)| \leq |f(x)-f(\alpha)| + |f(\alpha)| \leq \var_{L_k^+} f + \sup_{L^+_{K}} |f| \leq $$
and since $x \leq \frac{1}{k-1+\alpha}$
$$ \leq k ^\delta \Vert f \Vert_{K, \delta} + \var_{L^+_{K}} f + \frac{ \Vert f \Vert_{L^1(I_\alpha)}}{m(L^+_{K})} \leq \left( \left( \frac{1}{|x|}+1 \right)^{\delta} + K^\delta + \frac{1}{m(L^+_{K})} \right)\Vert f \Vert_{K, \delta}$$
Similarly for $x < 0$.
\end{proof}

Moreover, just as in the case of $BV$, the inclusion $B_{K, \delta} \rightarrow L^1$ is compact.

\begin{proposition} \label{compactlog}
For every $K$ sufficiently large, $\delta >0$, the unit ball 
$$ \mathcal{B} = \{ f \in B_{K, \delta} ,  \Vert f \Vert_{K, \delta} \leq 1 \}$$
is compact in the $L^1$-topology.
\end{proposition}

\begin{proof}
This fact is well-known when you consider $BV$ instead of $B_{K, \delta}$. Now, given $\{f_n\} \subseteq \mathcal{B}$, for any $k$ the sequence of restrictions $f_n\mid_{L_k}$ sits inside a closed ball in $BV(L_k)$ hence it has a subsequence which converges in $L^1(L_k)$ to some $F_k \in BV(L_k)$. By refining the subsequence as $k \rightarrow \infty$, one finds a subsequence $f_{n_l} \in \mathcal{B}$ such that for every $k$, $f_{n_l} \mid_{L_k} \rightarrow F_k$ in $L^1(L_k)$ and a.e. for $l \rightarrow \infty$. By uniqueness of the limit there exists $F$ such that $F \mid_{L_k} = F_k$. By lower semicontinuity of total variation, $k^{-\delta} \var_{L_k} F + \Vert F \Vert_{L^1(L_k)} \leq 1$, so $F \in \mathcal{B}$. We are just left with proving $f_{n_l} \rightarrow F$ in $L^1(I_\alpha)$ for $l \rightarrow \infty$. By lemma \ref{powergrowth} 
\begin{small}
$$ \int_{I_\alpha} |f_{n_l} - F| \leq \int_{I_\alpha \setminus L_k} |f_{n_l}| + |F| + \int_{L_k} |f_{n_l} - F|  \leq 2 \int_{I_\alpha \setminus L_k} \frac{A}{|x|^\delta} dx + \int_{L_k} |f_{n_l} - F_k|$$
\end {small}
The first term tends to $0$ as $k \rightarrow \infty$ and the second does so for $l \rightarrow \infty$ as $k$ is fixed.
\end{proof}

\subsubsection{Spectral decomposition in $B_{K, \delta}$}

The goal of this section is to prove a spectral decomposition analogous to theorem \ref{spectraldec} in the space $B_{K, \delta}$, namely 

\begin{theorem} \label{decBkdelta}
For every $\alpha \in (0,1]$, $0 < \delta < 1$ and $K$ sufficiently large, the transfer operator $\Phi_\alpha : B_{K, \delta} 
\rightarrow B_{K, \delta}$ decomposes as 
$$\Phi_\alpha = \Pi_\alpha + \Psi_\alpha$$ 
where $\Pi_\alpha$ and $\Psi_\alpha$ are bounded linear, commuting operators on $B_{K, \delta}$, $\rho(\Psi_\alpha) < 1$ and 
$\Pi_\alpha$ is a projection onto a one-dimensional eigenspace.
\end{theorem}

The main ingredient to get the spectral decomposition is again a Lasota-Yorke type estimate:

\begin{proposition} \label{LYlog}
Let $\alpha \in \left(0, 1 \right]$, $0 < \delta < 1$. Then there exist $K \in \mathbb{N}$, $0 < \lambda < 1$, $C > 0$, $D > 0$ 
such that 
$$\Vert \Phi^n_\alpha(f)\Vert_{K, \delta}  \leq C \lambda^n \Vert  f  \Vert_{K, \delta}  + D \Vert f \Vert_{L^1}\qquad \forall f \in B_{K, \delta}$$
\end{proposition}

\begin{proof} First consider the case $\alpha < 1$. By formula \eqref{transfer} and lemma \ref{BVprop}, 4.
$$ \var_{L_k^+} \Phi^n_\alpha(f) \leq \sum_{j \in \mathcal{P}_n} \var_{T^n_\alpha(I_j) \cap L_k^+} \frac{f \circ \sigma_j}{|(T^n_\alpha)' \circ \sigma_j| } + 
2 \sup_{T^n_\alpha(I_j)\cap L_k^+} \left| \frac{f \circ \sigma_j}{|(T^n_\alpha)' \circ \sigma_j| } \right| \leq $$
and by lemma \ref{BVprop}, 1. and the fact that $T_\alpha^n: I_j \rightarrow T_\alpha^n(I_j)$ is a homeomorphism
\begin{small}
$$\leq \sum_{I_j \in \mathcal{P}_n} 3 \var_{I_j} \frac{f}{|(T^n_\alpha)'|} + 
\frac{ 2 \int_{I_j}|f(x)|dx}{m(T_\alpha^n(I_j) \cap L_k^+)} \leq 3 \sum_{I_j \in \mathcal{P}_n} \var_{I_j} (f g_{n, \alpha}) + 
\frac{2  \Vert f \Vert_1}{\inf_{j \in \mathcal{P}_n} \{ m(T_\alpha^n(I_j)\cap L_k^+) \} }$$
\end{small}
where the inf is taken over all non-empty intervals. Now, note that by lemma \ref{BVprop}, 3. and prop. \ref{classeC}, 2.
$$\sum_{I_j \in \mathcal{P}_n} \var_{I_j} (f g_{n, \alpha}) \leq 
 \sum_{I_j \in \mathcal{P}_n} \var_{I_j} f \sup_{I_j} g_{n, \alpha} + 
\frac{2 \Vert f \Vert_1 }{1-\gamma_\alpha}$$
hence we are left with only one term to estimate: in order to do so, we will split the sum in several parts, according to the filtration $L_k$:

\begin{small}
$$\sum_{I_j \in \mathcal{P}_n} \var_{I_j} f \sup_{I_j} g_{n, \alpha} \leq 
\left\Vert g_{n, \alpha} \right\Vert_\infty \sum_{\stackrel{I_j \in \mathcal{P}_n}{ I_j \subseteq L_k }} \var_{I_j} f + 
\sum_{\stackrel{h = 1}{\epsilon = \pm}}^\infty \sum_{\stackrel{I_j \in \mathcal{P}_n}{I_j \subseteq L^\epsilon_{(h+1)k}\setminus L^\epsilon_{hk}}} \var_{I_j}f \sup_{I_j} g_{n, \alpha} \leq $$
$$ \leq \gamma_{\alpha}^n \var_{L_k}f  + \sum_{\stackrel{h = 1}{\epsilon = \pm}}^\infty \var_{L^\epsilon_{(h+1)k}} f \sup_{L^\epsilon_{(h+1)k} \setminus L^\epsilon_{hk}} g_{n, \alpha} \leq 
\gamma_\alpha^n  k^\delta \Vert f \Vert_{K, \delta} + \sum_{h =1}^\infty \Vert f \Vert_{K, \delta} [(h+1)k]^\delta \sup_{L_{(h+1)k} \setminus L_{hk}} g_{1, \alpha}  \leq $$
\end{small}
and since $L_{(h+1)k} \setminus L_{hk} = \left[ -\frac{1}{hk + \alpha}, -\frac{1}{(h+1)k+\alpha} \right)\cup \left( \frac{1}{(h+1)k + \alpha}, \frac{1}{hk+\alpha} \right]$
$$ \leq \Vert f \Vert_{K, \delta} k^\delta \left(  \gamma_\alpha^n + \sum_{h =1}^\infty \frac{ (h+1)^\delta}{h^2 k^2} \right) \leq \Vert f \Vert_{K, \delta} k^\delta \left( \gamma_\alpha^n + \frac{M}{K^2} \right)$$
for some universal constant $M$ for all $k\geq K$. The same estimate holds for $\var_{L_k^-} \Phi^n_\alpha(f)$.
Moreover, for fixed $n$ and $\alpha$ the set $\{ T_\alpha^n(I_j) \mid I_j \in \mathcal{P}_n \}$ is finite, and since $L_k^+$ and $L_k^-$ are increasing sequences of intervals, 
$\inf \{ m(T_\alpha^n(I_j) \cap L_k^\pm) \ : \ I_j \in \mathcal{P}_n, \ k \geq K\}$ is bounded below by a positive constant,  
and for every $\alpha$ one can choose $n$ and $K$ such that 
$\lambda := 6\left( \gamma_\alpha^{n} + \frac{M}{K^2} \right)< 1$.
By combining all estimates, there exists a constant $D$ such that 
$$ \Vert \Phi_\alpha^{n}(f) \Vert_{K, \delta} \leq \lambda \Vert f \Vert_{K, \delta} + D \Vert f \Vert_1 \qquad \forall f \in B_{K, \delta}$$
and by iteration the claim follows. The case $\alpha = 1$ follows similarly; in this case $\Vert g_{1, \alpha} \Vert_\infty = 1$, but
proposition \ref{classeC} is replaced by
$$ \Vert g_{n, \alpha} \Vert_\infty \leq 4 \left( \frac{\sqrt{5}-1}{2} \right)^{2n -4} \qquad \Vert g_{n, \alpha}' \Vert_\infty \leq 2$$
\end{proof}


\textit{Proof of theorem \ref{decBkdelta}.} By propositions \ref{compactlog} and \ref{LYlog}, 
the transfer operators $\Phi_\alpha$ acting on $B_{K, \delta}$ satisify the hypotheses of 
Ionescu-Tulcea and Marinescu's theorem \cite{ITM}, hence we have a spectral decomposition of $\Phi_\alpha$
with a finite number of spectral projectors onto eigenvalues of unit modulus. Moreover, 
mixing of $T_\alpha$ still implies there is only one eigenvalue of modulus one and its eigenspace is one-dimensional. \qed

\medskip 
Note that since $BV(I_\alpha) \subseteq B_{K, \delta}$, the invariant density $\rho_\alpha$ previously obtained is still a fixed 
point of $\Phi_\alpha$, hence $\Pi_\alpha$ is nothing but projection onto $\mathbb{C}\rho_\alpha$.

\subsubsection{End of proof}

The proof of theorem \ref{CLTunbounded} now follows from standard application of the martingale central limit theorem. 
We will refer to the version given in (\cite{Viana}, Thm. 2.11). In order to adapt it to our situation, we need the following two lemmas: 

\begin{lemma} \label{Edecay}
Let $\alpha \in (0, 1]$, $0 < \delta < \frac{1}{2}$, and $K$ s.t. theorem \ref{decBkdelta} holds, 
and consider $f \in B_{K, \delta}$ with $\int_{I_\alpha} f d\mu_\alpha = 0$. Denote by $\mathcal{F}_0$ the Borel $\sigma$-algebra
 on $I_\alpha$ and $\mathcal{F}_n := T_\alpha^{-n}(\mathcal{F}_0)$. Then 
$$\sum_{n = 0}^{\infty} \Vert \mathbb{E}(f \mid \mathcal{F}_n) \Vert_{L^2(\mu_\alpha)} < +\infty$$
\end{lemma}

\begin{proof}
$$\Vert \mathbb{E}(f \mid \mathcal{F}_n) \Vert_{L^2(\mu_\alpha)} = \sup \left\{\int_{I_\alpha} (\psi \circ T_\alpha^n) f d\mu_\alpha: \psi \in L^2(\mu_\alpha), \Vert \psi \Vert_{L^2(\mu_\alpha)} = 1 \right\} = $$
$$= \sup\left\{\int_{I_\alpha} \psi \Phi_\alpha^n(f \rho_\alpha) dx : \psi \in L^\infty(\mu_\alpha), \Vert \psi \Vert_{L^2(\mu_\alpha)} = 1 \right\} \leq \frac{\Vert \Phi_\alpha^n(f \rho_\alpha) \Vert_{L^2(dx)}}{\sqrt{\inf \rho_\alpha}}$$
Now, by lemma \ref{powergrowth} and since $0 < \delta < \frac{1}{2}$, $\Vert \Phi_\alpha^n(f\rho_\alpha)\Vert_{L^2(dx)} \leq C \Vert \Phi^n_\alpha(f \rho_\alpha) \Vert_{K, \delta}$, and 
by theorem \ref{decBkdelta} $\Phi^n_\alpha(f\rho_\alpha) = \Psi^n(f \rho_\alpha)$ goes to $0$ exponentially fast in $B_{K, \delta}$-norm as $n \rightarrow \infty$.
\end{proof}

\begin{lemma} \label{Bcohom}
Let $f \in B_{K, \delta}$ real-valued, non-constant such that $\int_{I_\alpha} f d\mu_\alpha = 0$. Then there exists no function $u \in B_{K, \delta}$ such that
$$f = u - u \circ T_\alpha \quad \mu_\alpha-a.e.$$ 
\end{lemma}

\begin{proof}
Notice that $\mu_\alpha$ and Lebesgue measure are abs. continuous w.r.t. each other, hence measure zero sets are the same. 
Suppose there exists $u$ which satisfies the equation; then, $u \circ T_\alpha$ belongs to $B_{K, \delta}$. 
However, 
$$\var_{L_k^+}(u \circ T_\alpha) \geq \sum_{\stackrel{ I_j \textup{ full}}{j \leq k}} \var_{I_j} (u \circ T_\alpha) 
= \sum_{\stackrel{ I_j \textup{ full}}{j \leq k}} \var_{(\alpha-1, \alpha)} u = (k - j_{min})\var_{(\alpha-1, \alpha)} u$$
On the other hand, $\var_{L_k^+}(u\circ T_\alpha) \leq k^{\delta}\Vert u\circ T_\alpha \Vert_{K, \delta}$ with $\delta < 1$, which contradicts 
the previous estimate unless $\var_{(\alpha-1, \alpha)} u = 0$, i.e. $u$ is constant a.e.
\end{proof}

\textit{Proof of theorem \ref{CLTunbounded}.} 
We can assume $\int_{I_\alpha} f d\mu_\alpha = 0$. By (\cite{Viana}, Thm. 2.11) and lemma \ref{Edecay}, 
the claim follows unless there exists $u \in L^2(\mu_\alpha)$ such that
$$f =  u - u \circ T_\alpha \qquad \mu_\alpha-a.e.$$
If there exists such $u$, one can assume that $\int u \ d\mu_\alpha = 0$, and then, by the proof of (\cite{Viana}, Thm. 2.11), $u$ is given by 
$$ u = - \sum_{j = 1}^{\infty} \frac{\Phi_\alpha^j(f \rho_\alpha)}{\rho_\alpha}$$ 
where convergence of the series is in $L^2(\mu_\alpha)$.
By the spectral decomposition, $\sum_j \Phi^j(f \rho_\alpha)$ converges also in $B_{K, \delta} \subseteq L^2(\mu_\alpha)$. Moreover, since $\rho_\alpha$ 
is in $BV$ and is bounded from below, then $\frac{1}{\rho_\alpha}$ is in $BV$. Thus,  
$u$ lies in $B_{K, \delta}$, and this contradicts lemma \ref{Bcohom} unless $f$ is constant.  \qed

\medskip
Now, the function $x \mapsto \log|x|$ belongs to every $B_{K, \delta}$, therefore 

\begin{corollary}
For every $\alpha \in (0,1]$, the Birkhoff averages for the observable $\log|T_\alpha'(x)| = -2\log|x|$ distribute normally around the value $h(T_\alpha)$. 
\end{corollary}

\section{Stability of standard deviation} \label{last}

Having established the convergence of Birkhoff sums to a Gaussian distribution, we are now interested in analyzing how the standard deviation of this Gaussian varies when $\alpha$ varies.
 The question is motivated by the numerical simulations in \cite{CMPT}, section 2. We prove the following 

\begin{theorem} \label{contsigma}
Let $f : (-1, 1) \rightarrow \mathbb{R}$ of class $C^1$. For every $\alpha \in (0, 1)$ let us consider the variance 
$$\sigma_\alpha^2 := \lim_{n \rightarrow \infty} \int_{I_\alpha} \left( \frac{S_n f - n \int_{I_\alpha} f d\mu_\alpha}{\sqrt{n}} \right)^2 d\mu_\alpha$$
Then, for every $\underline{\alpha} \in (0, 1)$
$$\lim_{\alpha \rightarrow \underline{\alpha}} \sigma_{\alpha}^2 = \sigma_{\underline{\alpha}}^2.$$ 
\end{theorem}

The variance $\sigma_\alpha^2$ of the limit distribution is the second derivative of 
the eigenvalues $\lambda_0(\theta)$ of the perturbed transfer operators $\{\Phi_{\alpha, f, \theta}\}$ (see the discussion in section \ref{CLTBVsect}, and 
in particular eq. 2. after prop. \ref{variance}). In order to prove the theorem, we will
prove uniform convergence in $\alpha$ of the eigenvalues, via application of theorem \ref{KellerLiv} to the 
family $\{\Phi_{\alpha, f, \theta}\}_{\{|\alpha - \underline{\alpha}| < \epsilon, |\theta| < \epsilon , 
\Vert f - \underline{f} \Vert_\infty < \epsilon \}}$. 

Hypothesis (1) of thm. \ref{KellerLiv} is easily proved:

\begin{lemma} \label{condition1}
For any $C > 0$, there exists $M > 0$ such that 
$$\Vert \Phi_{\alpha, f, \theta}^n \Vert_1 \leq M^n \qquad \forall n \in \mathbb{N} \ \forall \alpha \in (0,1) \ \forall |\theta| < C$$
for every $f \in L^{\infty}(I_\alpha)$ s.t $\Vert f \Vert_\infty \leq C$.
\end{lemma}

\begin{proof}
 For $g \in BV$
$$\Vert \Phi_{\alpha, f, \theta}^n(g) \Vert_1 = \Vert \Phi_{\alpha}^n(e^{\theta S_n f}g) \Vert_1 
\leq \Vert e^{\theta S_n f}g \Vert_1 \leq \Vert e^{\theta S_n f} \Vert_{\infty} \Vert g\Vert_1  \leq 
e^{n |Re \theta| \Vert f \Vert_{\infty}} \Vert  g \Vert_1 $$
where we used that the unperturbed operators have unit norm on $L^1$.
\end{proof} 

Hypothesis (4) follows directly from lemma \ref{Kellerdist}; the precise statement, whose proof we omit, is the following:

\begin{lemma} \label{continuitymodule}
Let $\underline{\alpha} \in (0,1)$. Then there exist $\epsilon, C > 0$ such that for any $f, \underline{f} \in BV([0,1])$ s.t. $\Vert f - \underline{f} \Vert_\infty < \epsilon$, $\forall |\theta| < \epsilon$, $\forall |\alpha - \underline{\alpha}| < \epsilon$
$$ \vert \! \vert \! \vert \Phi_{\underline{\alpha}, \underline{f}, \theta} - \Phi_{\alpha, f, \theta} \vert \! \vert \! \vert \leq  C\left( |\alpha -\underline{\alpha}|^{1/2} + \Vert f - \underline{f} \Vert_\infty\right)$$ 
\end{lemma}

We now check condition (2), using the estimates in section \ref{LY} to get a Lasota-Yorke inequality which is uniform in both $\alpha$ and $\theta$.

\begin{proposition}
Let $\underline{\alpha} \in (0, 1)$. There exist $0 < \lambda <1$, $\epsilon$, $C_2, C_3$ such that
$$ \var_{I_\alpha} \Phi^n_{\alpha, f, \theta}(g) \leq C_2 \lambda^n \var_{I_\alpha} g + C_3 \Vert g \Vert_1 \quad \forall n \in \mathbb{N} $$
for every $\alpha \in (\underline{\alpha}-\epsilon, \underline{\alpha} + \epsilon)$, for every $|\theta| < \epsilon$ and for every $f \in C^1(I_\alpha)$
with $\Vert f \Vert_{C^1} \leq 1$.

\end{proposition}

\begin{proof}
Let us fix $g \in BV$. We have

$$\var_{I_\alpha} \Phi_{\alpha, f, \theta}^n(g) = \var_{I_\alpha} \Phi_{\alpha}^n(e^{\theta S_n f}g)  
\leq \var_{I_\alpha} (e^{\theta S_n f}g \cdot g_{n, \alpha}) = \sum_{j \in \mathcal{P}_n} \var_{\overline{I_j}} (e^{\theta S_n f}g \cdot g_{n, \alpha}) $$
Note that, since $g_{n, \alpha} \mid_{\partial I_j} = 0$,  
$$\var_{\overline{I_j}} (e^{\theta S_n f}gg_{n, \alpha}) \leq
\var_{I_j} (e^{\theta S_n f}gg_{n, \alpha}) + 
e^{n|\theta| \Vert f \Vert_\infty} \var_{\overline{I_j}} (gg_{n, \alpha}) 
$$
Now, by lemma \ref{BVprop}, 3
\begin{small}
$$\var_{I_j} (e^{\theta S_n f}g_{n, \alpha}g) = \var_{I_j} \frac{ e^{\theta S_n f}g}{|(T^n_{\alpha})'|} \leq 
\sup_{I_j} \left| \left(\frac{e^{\theta S_n f}}{(T_\alpha^n)'} \right)' \right| \int_{I_j} |g| + 
\sup_{I_j} \left| \frac{e^{\theta S_n f}}{(T_\alpha^n)'}\right| \var_{I_j} g$$
\end{small}
and by expanding the derivative
\begin{small}
$$ \left( \frac{e^{\theta S_n f}}{(T_\alpha^n)'} \right)' = \frac{ (e^{\theta S_n f})'}{(T_\alpha^n)'} + \left( \frac{1}{(T_\alpha^n)'} \right)' e^{\theta S_n f} = \frac{ e^{\theta S_n f} \theta \sum_{k = 0}^{n-1} (f' \circ T_\alpha^k)(T_\alpha^k)'}{(T_\alpha^n)'} 
+ \left( \frac{1}{(T_\alpha^n)'} \right)' e^{\theta S_n f} = $$
$$ = e^{\theta S_n f} \left[ \theta \sum_{k = 0}^{n-1} \frac{(f' \circ T_\alpha^k)}{[(T_\alpha^{n-k})'\circ T_\alpha^{k}]} + \left( \frac{1}{(T_\alpha^n)'} \right)' \right]  \leq 
 e^{n |\theta|  \Vert f \Vert_\infty} \frac{ |\theta| \Vert f' \Vert_{\infty} + 2}{1-\gamma_\alpha} $$
\end{small}
Moreover, by the estimates of proposition \ref{LY} (eq. \eqref{finalLY}), for each $\underline{\alpha} \in (0,1)$ and each $n$ there exist $\eta, D$ such that
$$\var_{I_\alpha} (g g_{n,\alpha}) \leq (2n+4)\gamma_\alpha^n \var_{I_\alpha} g + D \Vert g \Vert_1 \qquad \forall \alpha \in (\underline{\alpha}-\eta, \underline{\alpha} + \eta)$$
hence by combining all previous estimates
\small
$$\var_{I_\alpha} \Phi_{\alpha, f, \theta}^n(g) \leq (2n+5) e^{n |\theta|  \Vert f \Vert_\infty} \gamma_\alpha^n  \var_{I_\alpha} g +
 e^{n |\theta|  \Vert f \Vert_\infty} \left( \frac{ |\theta| \Vert f' \Vert_{\infty} + 2}{1-\gamma_\alpha} + D \right) \Vert g \Vert_1$$
\normalsize
and the claim follows by choosing some $n$ large enough and iterating.
\end{proof}

\begin{remark}
Notice this is the only place where we need $f \in C^1$. This is because, if $f \in BV$, $e^{\theta S_n f}$ will not in general be of bounded variation.
\end{remark}

We are now ready to draw consequences for the spectral decomposition: let us denote by $\lambda_{\alpha, f}(\theta)$ the eigenvalue of 
$\Phi_{\alpha, f, \theta}$ which is closest to $1$.

\begin{lemma} \label{convergence}
Let $\underline{\alpha} \in (0, 1)$, and suppose we have a family $\{f_\alpha\}_{\alpha \in (0,1)}$ of functions $f_\alpha : [0,1] \rightarrow \mathbb{R}$ of class $C^1$ for every $\alpha$ and such that
\begin{itemize}
 \item[-] $\Vert f_\alpha - f_{\underline{\alpha}} \Vert_\infty \rightarrow 0 \ \textup{for }\alpha \rightarrow \underline{\alpha}$
\item[-] $\sup_{\alpha \in (0,1)} \Vert f'_\alpha \Vert_\infty < \infty$
\end{itemize}
Then there exists $\epsilon > 0$ such that $\lambda_{\alpha, f_\alpha}(\theta)$ converges to $\lambda_{\underline{\alpha}, f_{\underline{\alpha}}}(\theta)$ on $|\theta| < \epsilon$ uniformly in $\theta$ as $\alpha \rightarrow \underline{\alpha}$. 
\end{lemma}

\begin{proof}
Let us fix $r \in (\lambda_0, 1)$ and $\delta$ such that $0 < \delta < \frac{1-r}{2}$. Then the projectors
\begin{equation} \label{proj}
\Pi_{\alpha, f_\alpha, \theta} := \frac{1}{2 \pi i} \oint_{\partial B(1, \delta)} (z- \Phi_{\alpha, f_\alpha, \theta})^{-1} dz
\end{equation}
are defined for $|\alpha - \underline{\alpha}| < \epsilon$ and $|\theta| < \epsilon$ for some $\epsilon$ and for $\delta$ sufficiently small rank($\Pi_{\alpha, f_\alpha, \theta}$) = rank($\Pi_{\underline{\alpha}, f_{\underline{\alpha}}, 0}$) = 1 (\cite{KellerLiverani}, cor. 3) so they are all projections on the $1$-dimensional eigenspace relative to the eigenvalue which is closest to $1$. By Dunford calculus we also have
\begin{equation} \label{lproj}
\lambda_{\alpha, f_\alpha}(\theta)\Pi_{\alpha, f_\alpha, \theta} = \Phi_{\alpha, f_\alpha, \theta} \Pi_{\alpha, f_\alpha, \theta} = \frac{1}{2 \pi i} \oint_{\partial B(1, \delta)} z (z - \Phi_{\alpha, f_\alpha, \theta})^{-1} dz
\end{equation}
By thm. \ref{KellerLiv} and proposition \ref{continuitymodule} there exists $C$ such that for $|\underline{\alpha}-\alpha| < \epsilon$ and $|\theta| < \epsilon$
$$ |\!|\!| (z-\Phi_{\alpha, f_\alpha, \theta})^{-1} - (z - \Phi_{\underline{\alpha}, f_{\underline{\alpha}}, \theta})^{-1} |\!|\!| \leq C \left( |\alpha - \underline{\alpha}|^{1/2} + \Vert f_\alpha - f_{\underline{\alpha}} \Vert_\infty \right)^{\eta}$$
with $\eta > 0$ fixed by thm. \ref{KellerLiv} so by eqns \ref{proj} and \ref{lproj}
$$|\lambda_{\alpha}(\theta) - \lambda_{\underline{\alpha}}(\theta)| = O(|\alpha-\underline{\alpha}|^{1/2} + \Vert f_{\alpha} - f_{\underline{\alpha}} \Vert_\infty )^{\eta}$$
uniformly in $\theta$ as $\alpha \rightarrow \underline{\alpha}$.
\end{proof}

\textit{Proof of theorem \ref{contsigma}.} 
Let $f_\alpha :[0,1] \rightarrow \mathbb{R}$ be $f_\alpha(x) := f(x+\alpha-1) - \int_{\alpha-1}^\alpha f d\mu_\alpha$. 
Since $\tilde{\rho}_\alpha \rightarrow \tilde{\rho}_{\underline{\alpha}}$ in $L^1$ and $f(x+\alpha-1) \rightarrow f(x + \underline{\alpha}-1)$ 
in $L^\infty$, we have $\int_{\alpha-1}^\alpha f d\mu_\alpha \rightarrow \int_{\underline{\alpha}-1}^{\underline{\alpha}} f d\mu_{\underline{\alpha}}$,
 and the family $\{f_\alpha\}$ satisfies the hypotheses of lemma \ref{convergence}, therefore 
$\lambda_{\alpha, f_\alpha}(\theta)$ converges uniformly in a nbd of $\theta = 0$ to 
$\lambda_{\underline{\alpha}, f_{\underline{\alpha}}}(\theta)$. Since all $\lambda_{\alpha, f_\alpha}(\theta)$ are analytic in 
$\theta$ you also have convergence of all derivatives, in particular $\lambda_{\alpha, f_\alpha}''(0) \rightarrow \lambda_{\underline{\alpha}, f_{\underline{\alpha}}}''(0)$. 
We now note that $\int_0^1 f_\alpha(x) \tilde{\rho}_\alpha(x) dx = 0$, which implies, as we have seen in section \ref{spectral},
 that $\lambda''_{\alpha, f_\alpha}(0) = \sigma_{\alpha}^2$.  \qed

\section{Appendix}

Let us recall a few well-known properties of total variation:

\begin{lemma} \label{BVprop}
Let $I \subseteq \mathbb{R}$ be a bounded interval, $J \subseteq I$ a subinterval and $f$ of bounded variation. Then:
\begin{enumerate}
\item $$\sup_{x \in J} |f(x)| \leq \var_{J} f + \frac{1}{m(J)} \int_J |f(x)|dx$$
\item If $g \in BV(J)$, $$\var_J (fg) \leq \sup_{x\in J} |f(x)| \var_J g + \sup_{x \in J} |g(x)| \var_J f$$
\item If $g$ is of class $C^1$ on $J$, $$\var_J(fg) \leq \var_J f \sup_{x \in J}|g(x)| + \sup_{x \in J}|g'(x)|\int_J |f(x)| dx$$
\item  
$$\var_I(f \chi_{J}) \leq \var_J f + 2 \sup_{J} |f|$$
\end{enumerate}
\end{lemma}

Let us also prove the basic properties of $T_\alpha$ mentioned in section \ref{cyl}.
\textit{Proof of proposition \ref{classeC}.} 
 \begin{enumerate}
  \item 
$$  \sup_{j \in \mathcal{P}_1} \sup_{x \in I_j} \frac{1}{|(T_{\alpha})'(x)|} = \sup_{x \in [\alpha-1, \alpha]} x^2 \leq \max\{\alpha^2, (\alpha-1)^2 \} $$
The case for $n > 1$ follows from the chain rule for derivatives.  
\item Let $K_n := \sup_{j \in \mathcal{P}_n} \sup_{x \in I_j} \left| g_{n, \alpha}' (x) \right| $. For $n = 1$,
$$K_1 = \sup_{j \in \mathcal{P}_1} \sup_{x \in I_j} \left| \left( \frac{1}{(T_{\alpha})'} \right)'(x) \right| = \sup_{x \in [\alpha-1, \alpha] } 2|x| \leq 2$$ 

Now, 
$$(T_{\alpha}^{n+1})'(x) = (T_{\alpha}^n)'(T_{\alpha}(x))T_{\alpha}'(x)$$
$$(T_{\alpha}^{n+1})''(x) = (T_{\alpha}^n)''(T_{\alpha}(x))[T_{\alpha}'(x)]^2 + (T_{\alpha}^n)'(T_{\alpha}(x))T_{\alpha}''(x)$$
For every $x$ in the interior of some interval $I_j \in \mathcal{P}_{n+1}$, 
$$\left|\frac{(T_{\alpha}^{n+1})''(x)}{[(T_{\alpha}^{n+1})'(x)]^2}\right| \leq \left|\frac{(T_{\alpha}^n)''(T_{\alpha}(x))(T_{\alpha}'(x))^2}{[(T_{\alpha}^n)'(T_{\alpha}(x))T_{\alpha}'(x)]^2} + \frac{(T_{\alpha}^n)'(T_{\alpha}(x))T_{\alpha}''(x)}{[(T_{\alpha}^n)'(T_{\alpha}(x))T_{\alpha}'(x)]^2}\right| \leq $$
$$\leq \left|\frac{(T_{\alpha}^n)''(T_{\alpha}(x))}{[(T_{\alpha}^n)'(T_{\alpha}(x))]^2}\right| + \left|\frac{T_{\alpha}''(x)}{[T_{\alpha}'(x)]^2}\right| \frac{1}{|(T_{\alpha}^n)'(T_{\alpha}(x))|} \leq K_n + K_1 {\gamma_\alpha^n}$$
hence $K_{n+1} \leq K_n + 2{\gamma_\alpha^n}$ and by induction $K_{n} \leq \sum_{k = 0}^{n-1} 2{\gamma_\alpha^k} \leq \frac{2}{1-\gamma_\alpha}$.
\item
By induction on $n$: let $I^{-}_{j_M}$ be the interval of the partition $\mathcal{P}_1$ which contains $\alpha -1$ and $I^{+}_{j_m}$ be the one which contains $\alpha$. \\ For $n = 1$, $T_{\alpha}(I_j)= I_{\alpha}$ for $I_j \neq I^{-}_{j_M}, I^{+}_{j_m}$, hence 

$$\{ T_{\alpha}(I_j) | I_j \in \mathcal{P}_1 \} \subseteq \{ I_{\alpha}, T_{\alpha}(I^{-}_{j_M}), T_{\alpha}(I^{+}_{j_m}) \} $$

Let $n > 1$; consider an element of the partition $\mathcal{P}_{n+1}$, which will be of the form $I_{j_0} \cap T^{-1}(I_{j_1}) \cap \dots \cap T^{-n}(I_{j_n}) \neq \emptyset$, with $I_{j_0}, \dots, I_{j_n} \in \mathcal{P}_1$. If we let $L := I_{j_1} \cap \dots \cap T^{-(n-1)}(I_{j_n})$, we have $L \neq \emptyset$ and $L \in \mathcal{P}_n$. Moreover, one verifies that
$$T_{\alpha}^{n+1}(I_{j_0} \cap T_{\alpha}^{-1}(I_{j_1}) \cap \dots \cap T_{\alpha}^{-n}(I_{j_n})) \subseteq T_{\alpha}^n(I_{j_1} \cap \dots \cap T_{\alpha}^{-(n-1)}(I_{j_n})) = T_{\alpha}^n(L)$$
At this point we have two cases:
\begin{itemize}
 \item if $T_{\alpha}(I_{j_0}) \supseteq L$ then
$$T_{\alpha}^{n+1}(I_{j_0} \cap T_{\alpha}^{-1}(I_{j_1}) \cap \dots \cap T_{\alpha}^{-n}(I_{j_n})) = T_{\alpha}^n(L)$$
\item otherwise we have $T_{\alpha}(I_{j_0}) \nsupseteq L$ but $T_{\alpha}(I_{j_0}) \cap L \neq \emptyset$ (if the intersection is empty, so it is the interval we started with); since $T_{\alpha}(I_j) = I$ for $I_j \neq I^{-}_{j_M}, I^{+}_{j_m}$, this implies $I_{j_0} \in \{I^{-}_{j_M}, I^{+}_{j_m}\}$. Moreover, because $T_{\alpha}(I^{-}_{j_M})$ and $T_ {\alpha}(I^{+}_{j_m})$ are intervals with supremum equal to $\alpha$, there exists at most one interval $I_{\mu}$ of the partition $\mathcal{P}_n$ s.t. $T_{\alpha}(I^{+}_{j_m}) \cap I_{\mu} \neq \emptyset$ and $T_{\alpha}(I^{+}_{j_m}) \nsupseteq I_{\mu}$; in the same way there exists only one interval $I_{\nu}$ of the partition $\mathcal{P}_n$ such that $T_{\alpha}(I^{-}_{j_M}) \cap I_{\nu} \neq \emptyset$ and $T_{\alpha}(I^{-}_{j_M}) \nsupseteq I_{\nu}$, therefore either $L = I_{\mu}$ or $L = I_{\nu}$.
\end{itemize}
In conclusion $\{ T_{\alpha}^{n+1}(I_{j}) \mid I_j \in \mathcal{P}_{n+1} \}$ is contained in 
\begin{small}
$$\{ T_{\alpha}^n(I_j) \mid I_j \in \mathcal{P}_{n} \} \cup T_{\alpha}^{n+1}(I^{+}_{j_m} \cap T_{\alpha}^{-1}(I_{\mu})) \cup T_{\alpha}^{n+1}(I^{-}_{j_M} \cap T_{\alpha}^{-1}(I_{\nu}))$$
\end{small}
hence at every step the cardinality can only increase by at most $2$.
\item Recall that
$$g_{1, \alpha}(x) := \left\{ 
\begin{array}{ll} x^2 & \textup{if }x \textup{ belongs to some }I_j \\
		  0 & \textup{otherwise}
 
\end{array}
\right.$$
hence the claim follows from sommability of the series $\sum \frac{1}{k^2}$.
\end{enumerate}
\qed

\medskip

\noindent \textsc{Department of Mathematics, Harvard University}
\newline \textsc{One Oxford Street, Cambridge MA 02138 USA}
\newline{e-mail: tiozzo@math.harvard.edu}

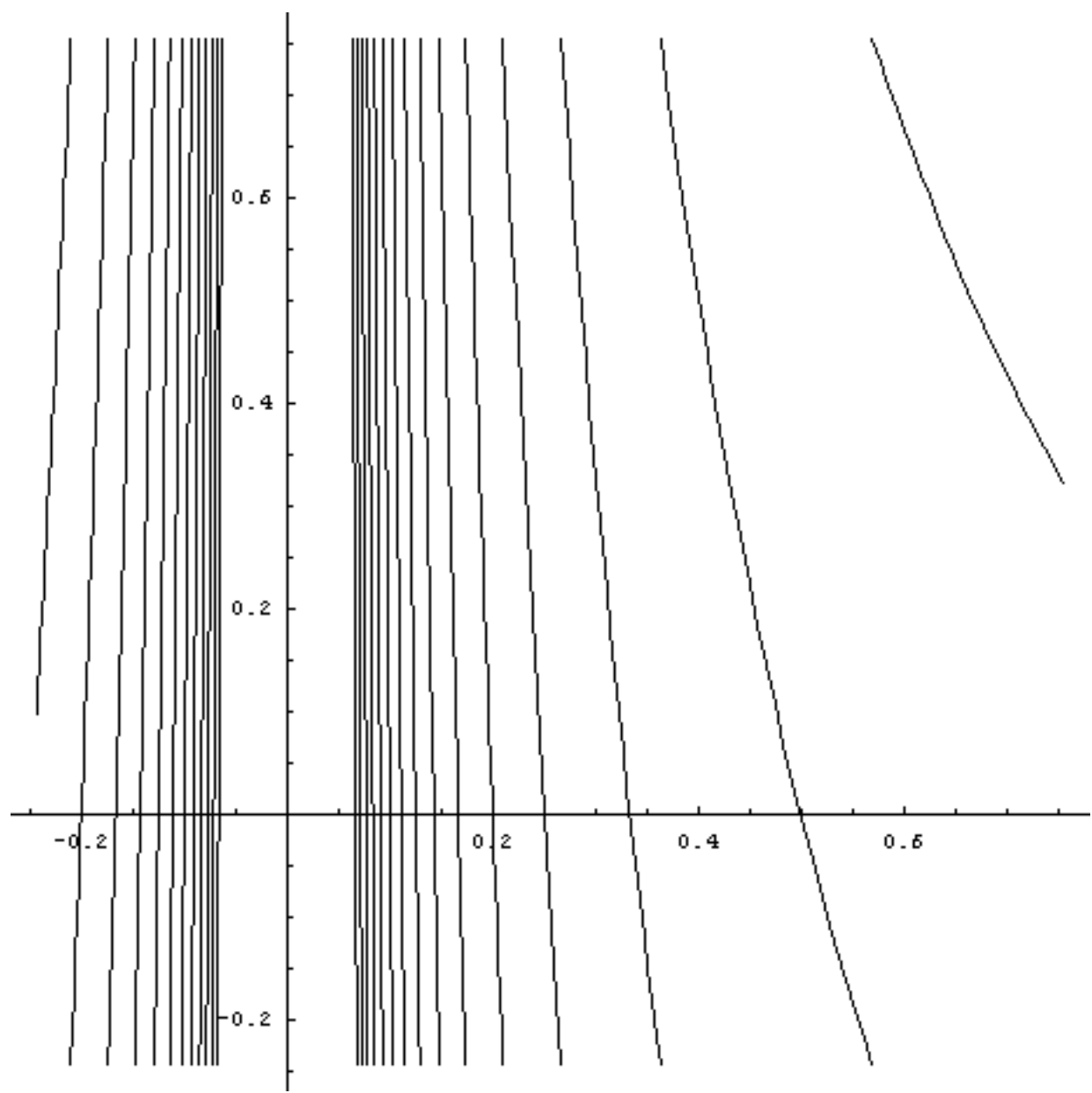
\end{document}